\pdfoutput=1
\documentclass[a4paper,english]{jnsao} 
\usepackage[utf8]{inputenc}
\usepackage[T1]{fontenc}
\usepackage{csquotes}
\usepackage{babel}
\usepackage{tikz}
\usepackage{pgfplots}
\usepgfplotslibrary{colorbrewer}
\usepackage{float}
\usepackage{enumitem}
\usepackage{comment}
\usepackage{mathtools}
\usepackage[nameinlink,capitalize]{cleveref}

\newcommand{\term}{\emph}

\newcommand{\field}[1]{\mathbb{#1}}
\newcommand{\N}{\mathbb{N}}

\newcommand{\R}{\field{R}}
\newcommand{\extR}{\overline \R}
\newcommand{\B}{B}

\newcommand{\norm}[1]{\|#1\|}

\newcommand{\abs}[1]{|#1|}

\newcommand{\inv}[1]{#1^{-1}}
\newcommand{\grad}{\nabla}
\newcommand{\freevar}{\,\boldsymbol\cdot\,}

\newcommand{\Union}\bigcup
\newcommand{\Isect}\bigcap
\newcommand{\union}\cup
\newcommand{\isect}\cap
\newcommand{\bigunion}\bigcup
\newcommand{\bigisect}\bigcap

\newcommand{\defeq}{:=}

\newcommand{\downto}{\searrow}
\newcommand{\upto}{\nearrow}
\newcommand{\subdiff}{\partial}

\DeclareMathOperator*{\argmin}{arg\,min}

\DeclareMathOperator{\ri}{ri}

\DeclareMathOperator{\sign}{sign}
\DeclareMathOperator{\Sign}{Sign}

\DeclareMathOperator{\divergence}{div}
\DeclareMathOperator{\range}{ran}

\def \uminusSym{\setbox0=\hbox{$\cup$}\rlap{\hbox 
        to\wd0{\hss\raise0.5ex\hbox{$\scriptscriptstyle{-}$}\hss}}\box0}
    
\newcommand{\iprod}[2]{\langle #1,#2\rangle}

\def\llangle{\langle\kern-3pt\langle}
\def\rrangle{\rangle\kern-3pt\rangle}

\newcommand{\weakto}{\mathrel{\rightharpoonup}}
\def \weaktostarSym{\setbox0=\hbox{$\rightharpoonup$}\rlap{\hbox 
        to\wd0{\hss\raise1ex\hbox{$\scriptscriptstyle{*\,}$}\hss}}\box0}
    
\def\linear{\mathbb{L}}

\newcommand{\setto}{\rightrightarrows}

\def\extR{\overline \R}

\def\opt#1{\bar #1}

\def\optx{{\opt{x}}}

\def\opty{{\opt{y}}}
\def\optd{{\opt{d}}}

\def\gap{\mathcal{G}}
\def\GenGap{\mathcal{G}}

\DeclareMathOperator{\dist}{dist}
\DeclareMathOperator{\prox}{prox}

\def\d{\,d}

\def\dualprod#1#2{\langle #1|#2\rangle}

\newcommand{\Meas}{\mathcal{M}}
\newcommand{\Lebesgue}{\mathcal{L}}

\let\phi=\varphi
\let\epsilon=\varepsilon
\def\alt#1{\tilde #1}

\def\Id{\mathop{\mathrm{Id}}}

\def\BVspace{\mathop{\mathrm{BV}}}

\renewrobustcmd{\downto}{{{\mathchoice%
            {\rotatebox[origin=c]{-20}{$\to$}}
            {\rotatebox[origin=c]{-20}{$\to$}}
            {\rotatebox[origin=c]{-20}{\scalebox{0.75}{$\to$}}}
            {\rotatebox[origin=c]{-20}{\scalebox{0.6}{$\to$}}}
}}}

\renewrobustcmd{\upto}{{{\mathchoice%
            {\rotatebox[origin=c]{20}{$\to$}}
            {\rotatebox[origin=c]{20}{$\to$}}
            {\rotatebox[origin=c]{20}{\scalebox{0.75}{$\to$}}}
            {\rotatebox[origin=c]{20}{\scalebox{0.6}{$\to$}}}
}}}

\manuscriptcopyright{}
\manuscriptlicense{}
\manuscripteprinttype{arXiv}
\manuscripteprint{2011.07575}

\theoremstyle{definition}
\newtheorem{assumption}[theorem]{Assumption}
\crefname{assumption}{Assumption}{Assumptions}

\date{2020-10-30 (revised 2021-01-17)}
\author{
    Tuomo Valkonen\thanks{Department of Mathematics and Statistics, University of Helsinki, Finland \emph{and} ModeMat, Escuela Politécnica Nacional, Quito, Ecuador, \email{tuomo.valkonen@iki.fi}, \orcid{0000-0001-6683-3572}}
    }
\shortauthor{T.~Valkonen}
\title{Regularisation, optimisation, subregularity}

\acknowledgements{%
    This research has been supported by the Academy of Finland grants 314701 and 320022.%
}

\begin{document}

\maketitle

\begin{abstract}
    Regularisation theory in Banach spaces, and non--norm-squared regularisation even in finite dimensions, generally relies upon Bregman divergences to replace norm convergence. This is comparable to the extension of first-order optimisation methods to Banach spaces.
    Bregman divergences can, however, be somewhat suboptimal in terms of descriptiveness.
    Using the concept of \term{(strong) metric subregularity}, previously used to prove the fast local convergence of optimisation methods, we show norm convergence in Banach spaces and for non--norm-squared regularisation. For problems such as total variation regularised image reconstruction, the metric subregularity reduces to a geometric condition on the ground truth: flat areas in the ground truth have to compensate for the fidelity term not having second-order growth within the kernel of the forward operator.
    Our approach to proving such regularisation results is based on optimisation formulations of inverse problems. As a side result of the regularisation theory that we develop, we provide \term{regularisation complexity} results for optimisation methods: how many steps $N_\delta$ of the algorithm do we have to take for the approximate solutions to converge as the corruption level $\delta \downto 0$?
\end{abstract}

\section{Introduction}

Let $A \in C^1(X; Y)$ between a Banach space $X$ of unknowns and a Hilbert space $Y$ of measurements.
A common  approach to solving the inverse problem $A(x)=\hat b$, having access only to corrupted measurements $b^\delta$ of the true data $\hat b$, is to solve for some regularisation parameter $\alpha_\delta>0$ and a convex regularisation functional $R$ the Tikhonov-style regularised problem
\begin{equation}
    \label{eq:intro:problem}
    \min_{x \in X} J_\delta(x) + \alpha_\delta R(x)
    \quad\text{where}\quad
    J_\delta(x) \defeq \frac{1}{2}\norm{A(x)-b_\delta}_Y^2.
\end{equation}
We then want to know whether solutions $x^\delta$ to this problem converge to a solution (or ground-truth) of the original equation as $\delta \downto 0$. One typically fixes a specific solution
\begin{equation}
    \label{eq:intro:ground-truth}
    \hat x \in \argmin_x \{ R(x) \mid A(x)=\hat b \}.
\end{equation}
Conventional regularisation theory \cite{engl2000regularization} treats the case $R(x)=\frac{1}{2}\norm{x}_X^2$ with $X$ also a Hilbert space. In this case $\hat x$ is the minimum-norm solution. Norm convergence $x^\delta \to \hat x$ can be shown provided $\hat x \in \range A'(\hat x)^*$ and $\alpha_\delta \downto 0$ as well as $\delta^2/\alpha_\delta \downto 0$ as $\delta \downto 0$.

How about Banach spaces $X$, and more complicated regularisation functionals $R$, such as non-differentiable sparsity-inducing $L^1$-norm regularisation, total variation regularisation, and its generalisations \cite{bredies2009tgv}?
Let $C \defeq \{x \in X \mid A(x)=\hat b\}$.
The first-order optimality conditions\footnote{Necessary and sufficient if $R$ is convex, proper, and lower semicontinuous, and $A$ is linear. More generally necessary; see \cref{sec:nonlinear}.} or Fermat principle for \eqref{eq:intro:ground-truth} requires that $0 \in \subdiff[\delta_C + \alpha_\delta \subdiff R](\hat x)$, where $\subdiff$ denotes the convex subdifferential, and $\delta_C$ the $\{0,\infty\}$-valued indicator function of $C$. This condition can for some $w \in Y^*$ be expanded as
\[
    0 \in A'(\hat x)^*w + \subdiff R(\hat x),
\]
which is known in the inverse problems literature as a \term{source condition}.
It  encodes the existence of an $R$-minimising ground-truth.
If the source condition holds, then it is shown in \cite{burger2004convergence} that the \term{Bregman divergence}
\begin{equation}
    \label{eq:intro:bregman}
    B_R^{\opt d}(x, \optx) \defeq \dualprod{\opt d}{\optx-x} + R(x) - R(\optx)
    \quad (x, \optx \in X;\, \opt d \in \subdiff R(\opt x))
\end{equation}
satisfies $B_R^{-A'(\hat x)^*w}(x^\delta, \hat x) \to 0$ provided the noise and regularisation parameters $\delta>0$ and $\alpha_\delta>0$ convergence as in the conventional norm-squared case. We refer to \cite{schuster2012regularization} for  the use of Bregman divergences in Banach space regularisation theory.

Convergence of Bregman divergences is, however, a relatively weak result. It does not imply norm convergence unless the Bregman divergence is \term{elliptic} in the words of \cite{tuomov-firstorder}, i.e., $B_R^{\opt d}(x, \optx) \ge \gamma\norm{x-\optx}^2$ for some $\gamma>0$. This simply means that $R$ is strongly convex, something that is not satisfied by $L^1$ and total variation regularisation.
Often, however, not the regulariser itself but the entire objective of \eqref{eq:intro:problem} admits a type of \emph{local} strong convexity.

\begin{example}
    On $\R^2$, take $A(x_1, x_2)=x_1$ and $R(x_1, x_2)=\abs{x_2}$.
    Then, due to the growth properties of the absolute value function at zero, $J_\delta+\alpha_\delta R$ is locally strongly convex at $(x_1, 0)$, but not at $(x_1, x_2)$ for $x_2 \ne 0$.
\end{example}

\noindent
If $A$ is linear, then $J_\delta+\alpha_\delta R$ is convex. In this case local strong convexity is equivalent to the \term{strong metric subregularity} of the subdifferential $\subdiff[J_\delta(x) + \alpha_\delta R]$ \cite{artacho2013metric,aragon2008characterization}.
Strong metric subregularity, and (non-strong) metric subregularity introduced in \cite{ioffe1979regular,dontchev2004regularity}, are Lipschitz-like properties of set-valued maps.
We briefly recall their definitions and interpretations in \cref{sec:subregularity}, referring to \cite{rockafellar-wets-va,ioffe2017variational,clasonvalkonen2020nonsmooth} for more comprehensive introductions. To see how such concepts can be used in optimisation, we refer to \cite{tuomov-subreg}.

We will exploit strong metric subregularity and an intermediate concept between strong and non-strong metric subregularity to show for general regularisers $R$ the norm convergence of \emph{approximate} solutions to \eqref{eq:intro:problem}. We do this first in \cref{sec:linear} for linear inverse problems, and then in \cref{sec:nonlinear} for nonlinear inverse problems. In \cref{sec:nonlinear} we also generalise the results to general data discrepancies $E$ in place of the squared norm.
In \cref{sec:examples} we verify the relevant subregularity---expressed as a \term{strong source condition}---for $\ell^1$ regularisation in finite dimensions, and for total variation regularisation in $L^2(\Omega)$. In both cases, the lack of ellipticity of $A^*A$  (i.e., $A^*A \not\ge \gamma \Id$ for some $\gamma>0$) has to be compensated for by the regularisation term.
In the first case, we obtain this unconditionally, whereas for total variation regularisation our results are more preliminary and require the ground-truth to have “strictly flat areas” that perform this compensation.

We finish in \cref{sec:complexity} with interpretations of the regularisation results of \cref{sec:linear} as “regularisation complexity” results for optimisation methods: forward-backward splitting and primal-dual proximal splitting. We derive expressions for how many iterations $N_\delta$ of the algorithm are needed to produce  approximate solutions $x^\delta$ to \eqref{eq:intro:problem} that converge to $\hat x$ as $\delta \downto 0$.

Our proof approach is intrinsically based around treating inverse problems as optimisation problems. In this sense, our work is closely related to \cite{hofmann2007convergence}, which also proves convergence in a “hybrid” topology involving the Bregman divergence and the norm in the data space: $B_R^{-A'(\hat x)^*w}(x^\delta, \hat x) + \norm{A(x^\delta)-A(\hat x)} \to 0$.
More general optimisation-based formulations of inverse problems are treated in \cite{kaltenbacher2018minimization}, whereas the weak-$*$ convergence of solutions to total generalised variation regularised linear inverse problems is treated in \cite{bredies2014regularization}.
Norm convergence of solutions to multibang control problems is studied in \cite{do2019discrete}. There the metric subregularity $R$ is employed, however, not that of the entire Tikhonov-regularised objective as we will do.

Iterative regularisation methods \cite{kaltenbacher2008iterative} also closely tie optimisation methods to regularisation, however, this is different from the “regularisation complexity” results of \cref{sec:complexity}: iterative regularisation adapts the parameter $\alpha_\delta>0$ within each step of the optimisation method, whereas we simply want to know how many steps of the method we need to take for fixed $\alpha_\delta$. Numerically efficient iterative regularisation is largely limited to smooth regularisers $R$ as nonsmooth approaches require solving a difficult “inner” optimisation problem in each step of the “outer” method \cite{bachmayr2009iterative}.

Finally, metric regularity---a concept stronger than metric subregularity and distinct from strong metric subregularity---has been used in \cite{gaydu2011stability} to study the regularisation of set-valued inclusions $0 \in T(x)$ via the Tikhonov-style expression $0 \in [T+\alpha_\delta \Id](x^\delta)$. The incorporation of the identity map $\Id$ requires $T$ to be a set-valued map form $X$ to $X$, whereas subdifferentials are set-valued maps from $X$ to $X^*$.

\subsection*{Notation and elementary results}

We denote the extended reals by $\extR \defeq [-\infty,\infty]$.
We write $H: X \setto Y$ when $H$ is a set-valued map from the space $X$ to $Y$.
For Fréchet differentiable $F:X \to R$, we write $F'(x) \in X^*$ for the Fréchet derivative at $x \in X$. Here $X^*$ is the dual space to $X$.
For a convex function $F: X \to \extR$, we write $\subdiff F: X \setto X^*$ for its subdifferential map.
On a normed space $X$, for a point $x \in X$ and a set $U \subset X$, we write $\dist(x, U) \defeq \inf_{x' \in U} \norm{x-x'}_X$, where $\norm{\freevar}_X$ is the norm on $X$. We also write $\dist^2(x, U) \defeq \dist(x, U)^2$.
We write $\iprod{x}{x'}$ for the inner product between two elements $x$ and $x'$ of a Hilbert space $X$, and $\dualprod{x^*}{x} \defeq x^*(x)$ for the dual product or dual pairing in a Banach space.
We write $\Id: X \to X$ for the identity operator on $X$ and $\delta_A: X \to \extR$ for the $\{0,\infty\}$-valued indicator function of a set $A \subset X$.

For $X$ a Hilbert space, we will frequently use Pythagoras' three-point identity
\begin{equation}
    \label{eq:intro:three-point}
    \iprod{x-y}{x-z}_X = \frac{1}{2}\norm{x-y}_X^2 - \frac{1}{2}\norm{y-z}_X^2 + \frac{1}{2}\norm{x-z}_X^2
    \quad
    (x,y,z\in X)
\end{equation}
and (inner product) Young's inequality
\begin{equation}
    \label{eq:intro:young}
    \iprod{x}{y}
    \le
    \norm{x}_X\norm{y}_X \le
    \frac{1}{2\alpha}\norm{x}_X^2 + \frac{\alpha}{2}\norm{y}_X^2
    \quad(x,y \in X,\, \alpha>0).
\end{equation}


\section{Metric subregularity and local subdifferentiability}
\label{sec:subregularity}

We recall that a set-valued function $H: X \setto Y$ is \term{metrically subregular} at $\optx$ for $\opty$ if $\opty \in H(\optx)$ and there exists a constant $\kappa>0$ as well as neighbourhoods $U \ni \optx$ and $V \ni \opty$ such that
\[
    \dist(x, \inv{H}(\opty)) \le \kappa \dist(\opty, H(x) \isect V)
    \quad (x \in U).
\]
If the stronger inequality
\begin{equation*}
    \norm{x-\optx}_X \le \kappa \dist(\opty, H(x) \isect V)
    \quad (x \in U)
\end{equation*}
holds, then we say that $H$ is \term{strongly metrically subregular} at $\optx$ for $\opty$.
The latter property can equivalently be stated as $H$ being metrically subregular with $\optx$ an isolated point of $\inv{H}(\opty)$.

We recall from \cite{aragon2008characterization,artacho2013metric} the following characterisations of the metric subregularity and strong metric subregularity of convex subdifferentials.

\begin{theorem}[{\cite[Theorem 3.3]{aragon2008characterization}}]
    \label{thm:subreg:convex}
    Let $F: X \to \extR$ be a convex function on a Banach space $X$, $\optx \in X$, and $\optx^* \in \subdiff F(\optx)$.
    Then $\subdiff F$ is metrically subregular at $\optx$ for $\optx^*$ if and only if there exists a neighbourhood $U$ of $\optx$ and a constant $\gamma>0$ such that
    \begin{equation}
        \label{eq:subreg:convex}
        F(x) \ge F(\optx) + \dualprod{\optx^*}{x-\optx} + \gamma \dist^2(x, \inv{[\subdiff F]}(\optx^*))
        \quad (x \in U).
    \end{equation}
    More precisely, \eqref{eq:subreg:convex} implies metric subregularity with $\kappa=1/\gamma$ while metric subregularity implies \eqref{eq:subreg:convex} for any $0 < \gamma < 1/(4\kappa)$.
\end{theorem}

\begin{theorem}[{\cite[Theorem 3.5]{aragon2008characterization}}]
    \label{thm:subreg:convex:strong}
    Let $F: X \to \extR$ be a convex function on a Banach space $X$, $\optx \in X$, and $\optx^* \in \subdiff F(\optx)$.
    Then $\subdiff F$ is strongly metrically subregular at $\optx$ for $\optx^*$ if and only if there exists a neighbourhood $U$ of $\optx$ and a constant $\gamma>0$ such that
    \begin{equation}
        \label{eq:subreg:convex:strong}
        F(x) \ge F(\optx) + \dualprod{\optx^*}{x-\optx} + \gamma \norm{x-\optx}_X^2
        \quad (x \in U).
    \end{equation}
    More precisely, \eqref{eq:subreg:convex:strong} implies strong metric subregularity with $\kappa=1/\gamma$ while strong metric subregularity implies \eqref{eq:subreg:convex:strong} for any $0 < \gamma < 1/(4\kappa)$.
\end{theorem}

We call the expressions \eqref{eq:subreg:convex} and \eqref{eq:subreg:convex:strong} the \term{semi-strong} and \term{strong local subdifferentiability} of $F$ at $\optx$ for $\optx^*$.
Compared to standard strong subdifferentiability\footnote{Equivalent to strong convexity in Hilbert spaces; see, e.g., \cite{clasonvalkonen2020nonsmooth}.}, both conditions localise the notion to a neighbourhood of $\optx$. Moreover, \eqref{eq:subreg:convex} in a sense squeezes the set $\inv{[\subdiff F]}(\optx^*)$ into a single point. Consider, for example, $F(x)=\max\{0,\abs{x}-1\}$ in $\R$. Let $\optx \in [-1, 1]$ be arbitrary and $\optx^*=0$. Then $\inv{[\subdiff F]}(\optx^*)=[-1, 1]$, so that \eqref{eq:subreg:convex} only requires $F$ to grow once $x$ leaves $[-1, 1] \ni \optx$ instead of immediately as it leaves $\{\optx\}$ as is required by \eqref{eq:subreg:convex:strong}.
In the context of optimisation methods, \eqref{eq:subreg:convex} is useful for the study of convergence of iterates to an entire set of minimisers of $F$, without singling out one, while \eqref{eq:subreg:convex:strong} is useful for studying the convergence to a known specific minimiser. In the following \cref{sec:linear}, we will likewise work both with a set of ground-truths $\hat X$ and a specific ground-truth $\hat x$.

Minding the locality of the definitions, using Clarke subdifferentials \cite{clarke1990optimization}, it is not difficult to extend \cref{thm:subreg:convex,thm:subreg:convex:strong} to locally convex functions, i.e., non-convex functions that locally satisfy a second-order growth condition. However, to get useful regularisation results, we need to replace (non-strong) metric subregularity with an intermediate and slightly relaxed version, motivated by the notions of partial submonotonicity and subregularity introduced in \cite{tuomov-subreg}. We therefore handle non-convex $F$ through direct definitions analogous to \eqref{eq:subreg:convex} and \eqref{eq:subreg:convex:strong}.

Namely, on Banach spaces $X$ and $Y$, let $A \in \linear(X; Y)$ and let $F: X \to \extR$ be a (possibly non-convex) function. Also pick $\mu>0$. Then we say that $f$ is \term{$(A, \mu)$-strongly locally subdifferentiable} at $\optx \in X$ for $\optx^* \in X^*$ with respect to a set $\opt X \subset X$ if there exists a neighbourhood $U$ of $\optx$ and a constant $\gamma>0$ such that
\begin{equation}
    \label{eq:subreg:strong-subdiff}
    F(x) \ge F(\optx) + \dualprod{\optx^*}{x-\optx} + \gamma \norm{A(x-\optx)}_Y^2 + \gamma \mu \dist^2(x, \opt X)
    \quad (x \in U).
\end{equation}
Typically $\opt X \ni \optx$.
We do not assume $\optx^*$ to be a subdifferential of $F$ at $\optx$ in any conventional sense.
The idea is that for $F(x)=\frac{1}{2}\norm{Ax-b_\delta}_Y^2 + \alpha_\delta R(x)$ as in \eqref{eq:intro:problem}, we will in $\range A^*$ exploit the available growth away from $\optx$, but otherwise only the growth away from the set $\opt X$.

We will in \cref{sec:examples} provide examples of the different forms of strong local subdifferentiability and strong metric subregularity.
Before this, in the next section, we will use them to derive regularisation theory.

\section{Linear inverse problems}
\label{sec:linear}

We now derive subregularity-based regularisation theory for linear inverse problems.
For all corruption parameters $\delta > 0$ and measured data $b_\delta \in Y$, we approximate solutions $\hat x \in X$ to the problem
\begin{equation}
    \label{eq:linear:problem:truth}
    A\hat x=\hat b,
\end{equation}
with unknown data $\hat b$, through the regularised problems
\begin{equation}
    \label{eq:linear:problem:regularised}
    \min_{x \in X} J_\delta(x) + \alpha_\delta R(x)
    \quad\text{where}\quad
    J_\delta(x) \defeq \frac{1}{2}\norm{Ax-b_\delta}_Y^2.
\end{equation}

\subsection{General assumptions and concepts}

We denote the set of $R$-minimising solutions to \eqref{eq:linear:problem:truth} by $\hat X$.
Under the assumptions that we state next, these are characterised exactly through the satisfaction for some $\hat w \in Y$ of the \term{basic source condition}
\begin{equation}
    \label{eq:linear:basic-source-condition}
    A\hat x=\hat b
    \quad\text{and}\quad
    A^*\hat w + \subdiff R(\hat x) \ni 0.
\end{equation}

\begin{assumption}[Corruption level and solution accuracy]
    \label{ass:linear:main}
    On a Banach space $X$ and a Hilbert space $Y$, the regularisation functional $R: X \to \extR$ is convex, proper, and lower semicontinuous, and $A \in \linear(X; Y)$.
    The parametrisation $\delta>0$ of the corruption level is such that the corrupted measurements $b_\delta \in Y$ of the ground-truth $\hat b \in Y$ satisfy
    \begin{equation}
        \label{eq:linear:noise}
        \norm{b_\delta-\hat b}_Y \le \delta
        \quad (\delta>0).
    \end{equation}
    Moreover, we assume that \eqref{eq:linear:problem:regularised} is solved to a given accuracy $e_\delta \ge 0$ in the sense that
    \begin{equation}
        \label{eq:linear:accuracy}
        [J_\delta + \alpha_\delta R](x_\delta) - [J_\delta+\alpha_\delta R](\hat x) \le e_\delta
        \quad (\delta>0)
    \end{equation}
    for any given $\hat x \in \hat X$.
\end{assumption}

It does not matter which $\hat x \in \hat X$ we use in the accuracy condition \cref{eq:linear:accuracy} since
\[
    J_\delta(\hat x) + \alpha_\delta R(\hat x)
    = \frac{1}{2}\norm{\hat b-b_\delta}^2 + \alpha_\delta R(\hat x)
    = J_\delta(\alt x) + \alpha_\delta R(\alt x)
    \quad (\hat x, \alt x \in \hat X).
\]
Also, since $\hat x$ is not in general a solution to \eqref{eq:linear:problem:regularised} for $\alpha_\delta > 0$, even with $e_\delta=0$, the problem \eqref{eq:linear:problem:regularised} may not have to be solved to full accuracy to satisfy \eqref{eq:linear:accuracy}

\subsection{A basic optimisation-based estimate}

To motivate our contributions, we start by deriving basic estimates based on Bregman divergencesThese are similar to those in \cite{burger2004convergence}, however allow for the inexact solution of \eqref{eq:linear:problem:regularised}.
Specifically, the following result shows that for convergence, as the corruption level $\delta \downto 0$, we need $\alpha_\delta \downto 0$, $\delta^2/\alpha_\delta \downto 0$, and $e_\delta/\alpha_\delta \downto 0$.

\begin{theorem}
    \label{thm:linear:bregman}
    Suppose \cref{ass:linear:main} holds and that $\hat x \in X$ satisfies for some $\hat w \in Y^*$ the basic source condition \eqref{eq:linear:basic-source-condition}.
    Then
    \[
        0 \le B_R^{-A^*\hat w}(x_\delta, \hat x)
        \le \frac{e_\delta}{\alpha_\delta} + \frac{\delta^2}{\alpha_\delta} + \alpha_\delta\norm{\hat w}_Y^2.
    \]
\end{theorem}

\begin{proof}
    By \cref{ass:linear:main} and expansion we have
    \[
        \begin{aligned}[t]
        e_\delta
        &
        \ge
        [J_\delta+\alpha_\delta R](x_\delta)-[J_\delta+\alpha_\delta R](\hat x)
        \\
        &
        =
        \frac{1}{2}\norm{Ax_\delta-b_\delta}_Y^2 - \frac{1}{2}\norm{A\hat x - b_\delta}_Y^2 + \alpha_\delta[R(x_\delta)-R(\hat x)].
        \end{aligned}
    \]
    Continuing with Pythagoras' three-point identity \eqref{eq:intro:three-point} we rearrange the squared norms to obtain
    \[
        e_\delta
        \ge
        \iprod{A\hat x-b_\delta}{A(x_\delta-\hat x)}
        + \frac{1}{2}\norm{A(x_\delta-\hat x)}_Y^2
        + \alpha_\delta[R(x_\delta)-R(\hat x)].
    \]
    Since $A\hat x=\hat b$, using the definition \eqref{eq:intro:bregman} of the Bregman divergence, this further rerarranges as
    \[
        e_\delta
        \ge
         \iprod{\hat b-b_\delta-\alpha_\delta \hat w}{A(x_\delta-\hat x)} + \frac{1}{2}\norm{A(x_\delta-\hat x)}_Y^2 +\alpha_\delta B_R^{-A^*\hat w}(x_\delta, \hat x).
    \]
    Now using Young's inequality \eqref{eq:intro:young} on the inner product term , we obtain
    \begin{equation}
        \label{eq:optim:jr-est}
        e_\delta
        \ge
        -\frac{1}{2}\norm{\hat b-b_\delta-\alpha_\delta \hat w}_Y^2 + \alpha_\delta B_R^{-A^*\hat w}(x_\delta, \hat x).
    \end{equation}
    Further estimating using Young's inequality that
    \begin{equation*}
        \frac{1}{2\alpha_\delta}\norm{\hat b-b_\delta-\alpha_\delta \hat w}_Y^2
        \le
        \frac{1}{\alpha_\delta}\norm{b_\delta-\hat b}_Y^2 + \alpha_\delta\norm{\hat w}_Y^2
        \le
        \frac{\delta^2}{\alpha_\delta} + \alpha_\delta\norm{\hat w}_Y^2,
    \end{equation*}
    we therefore obtain the claim.
\end{proof}

\begin{remark}
    In place of \eqref{eq:optim:jr-est} we could alternatively estimate
    \[
        \iprod{\hat b-b_\delta-\alpha_\delta \hat w}{A(x_\delta-\hat x)} + \frac{1}{2}\norm{A(x_\delta-\hat x)}_Y^2
        \ge
        -\frac{1}{2(1-\alpha_\delta)}\norm{\hat b-b_\delta-\alpha_\delta \hat w}_Y^2
        + \frac{\alpha_\delta}{2}\norm{A(x_\delta-\hat x)}_Y^2.
    \]
    Akin to the approach of \cite{hofmann2007convergence}, we would then obtain
    \[
        \frac{1}{2\alpha_\delta }\norm{A(x_\delta-\hat x)}_Y^2 + B_R^{-A^*\hat w}(x_\delta, \hat x)
        \le
        \frac{e_\delta}{\alpha_\delta} + \frac{\alpha_\delta}{1-\alpha_\delta}\norm{\hat w}_Y^2 + \frac{\delta^2}{\alpha_\delta(1-\alpha_\delta)}.
    \]
    Thus $A(x_\delta-\hat x) \to 0$ significantly faster than $B_R^{-A^*\hat w}(x_\delta, \hat x) \to 0$ provided, as before, that  $\alpha_\delta \to 0$, $\delta/\alpha_\delta \to 0$, and $e_\delta/\alpha_\delta \to 0$ as $\delta \downto 0$.
    This motivates our next results,  essentially combining, via strong metric subregularity, the two different convergences to yield norm convergence.
\end{remark}

\subsection{Estimates based on a strong source condition}

We start with the next lemma that we will {be used} to show that the approximate regularised solutions $x_\delta$ are close to $\hat x$ for small enough noise level, regularisation parameter, and accuracy parameter.

\begin{lemma}
    \label{lemma:linear:a-convergence}
    Suppose \cref{ass:linear:main} holds at $\hat x \in \hat X$. Then
    \[
        \norm{A(x_\delta-\hat x)}_Y^2 \le 4(e_\delta + \delta + \alpha_\delta R(\hat x))
    \]
    and
    \[
        R(x_\delta) \le R(\hat x) + \frac{e_\delta+\delta^{2}}{\alpha_\delta}.
    \]
\end{lemma}

\begin{proof}
    By \cref{ass:linear:main}, since $A\hat x=\hat b$, first using Young's inequality we have
    \[
        \begin{aligned}
        \frac{1}{2}\norm{A(x_\delta-\hat x)}_Y^2
        + 2\alpha_\delta R(x_\delta)
        &
        \le
        \norm{Ax_\delta-b_\delta}_Y^2
        + 2\alpha_\delta R(x_\delta)
        +\norm{b_\delta-\hat b}_Y^2
        \\
        &
        \le
        2e_\delta + 2\norm{b_\delta-\hat b}_Y^2 + 2\alpha_\delta R(\hat x)
        \\
        &
        \le 2(e_\delta + \delta^2 + \alpha_\delta R(\hat x)).
        \end{aligned}
    \]
    This finishes the proof.
\end{proof}

We will need the following “strong source condition” based on strong metric subregularity.

\begin{assumption}[Strong source condition]
    \label{ass:linear:strong-source-condition}
    Assume that $\hat x \in X$ satisfies for some $\hat w \in Y$ the basic source condition \eqref{eq:linear:basic-source-condition}.
    Moreover, for all $\delta>0$, for given $\alpha_\delta,\gamma_\delta>0$, assume that $J_\delta+\alpha_\delta R$ is strongly locally subdifferentiable at $\hat x$ for $J_\delta'(\hat x)-\alpha_\delta A^*\hat w$ with respect to the norm
    \[
        \norm{x}_\delta \defeq \sqrt{\norm{Ax}_Y^2 + \gamma_\delta \norm{x}_X^2} \quad (x \in X).
    \]
    The factor $\gamma>0$ of strong local subdifferentiability, as defined in \eqref{eq:subreg:strong-subdiff}, must be independent of $\delta>0$ and, for some $\rho>0$, we must have
    \begin{equation}
        \label{eq:linear:strong-subreg:u-cond}
        U \supset U_\rho \defeq \{x \in X \mid \norm{A(x-\hat x)} \le \rho,\, R(x) \le R(\hat x) + \rho\}.
    \end{equation}
    Then we say that $\hat x$ satisfies for $\hat w$ the \term{strong source condition}.
\end{assumption}

As we recall from \cref{thm:subreg:convex:strong} due to \cite{aragon2008characterization,artacho2013metric}, the strong local subdifferentiability required in \cref{ass:linear:strong-source-condition} is equivalent to the strong metric subregularity of $\subdiff[J_\delta+\alpha_\delta R]$, i.e., of $x \mapsto A^*(Ax-b^\delta)+\alpha_\delta \subdiff R(x)$ at $\hat x$ for $A^*(A\hat x-b_\delta)-\alpha_\delta A^*\hat w=A^*(\hat b-b_\delta-\alpha_\delta \hat w)$.

\begin{theorem}
    \label{thm:linear:strong-subreg}
    Suppose \cref{ass:linear:main} and the \emph{strong} source condition of \cref{ass:linear:strong-source-condition} hold at $\hat x$ for some $\hat w$.
    Suppose $(e_\delta+\delta^2)/\alpha_\delta \downto 0$ and $\alpha_\delta \downto 0$ as $\delta \downto 0$.
    Then there exists $\bar\delta>0$ such that if  $\delta \in (0, \bar\delta)$, we have
    \begin{equation}
        \label{eq:linear:strong-subreg:estimate}
        \norm{x_\delta-\hat x}_X^2
        \le
        \frac{e_\delta}{\gamma\gamma_\delta}
        +\frac{\delta^2}{2\gamma^2\gamma_\delta}
        +\frac{\alpha_\delta^2}{2\gamma^2\gamma_\delta}\norm{\hat w}_Y^2.
    \end{equation}
\end{theorem}

\begin{proof}
    We have
    \[
        J_\delta'(\hat x) - \alpha_\delta A^*\hat w
        =A^*(A\hat x-b_\delta-\alpha_\delta \hat w)
        =A^*(\hat b-b_\delta-\alpha_\delta \hat w).
    \]
    By the assumption that  $(e_\delta+\delta^2)/\alpha_\delta \downto 0$ and $\alpha_\delta \downto 0$ as $\delta \downto 0$, \cref{lemma:linear:a-convergence}, and \eqref{eq:linear:strong-subreg:u-cond} in \cref{ass:linear:strong-source-condition}, for suitably small $\delta>0$, we have $x_\delta \in U$.
    Hence by \cref{ass:linear:main,ass:linear:strong-source-condition} followed by Young's inequality \eqref{eq:intro:young},
    \begin{equation}
        \label{eq:linear:subreg:base-est}
        \begin{aligned}[t]
        e_\delta
        &
        \ge
        [J_\delta+\alpha_\delta R](x_\delta)-[J_\delta+\alpha_\delta R](\hat x)
        \\
        &
        \ge \dualprod{J_\delta'(\hat x) - \alpha_\delta  A^*\hat w}{x_\delta-\hat x}
        + \gamma\norm{x_\delta-\hat x}_\delta^2
        \\
        &
        = \iprod{\hat b-b_\delta-\alpha_\delta \hat w}{A(x_\delta-\hat x)}
        + \gamma\norm{A(x_\delta-\hat x)}_Y^2
        + \gamma\gamma_\delta\norm{x_\delta-\hat x}_X^2
        \\
        &
        \ge
        -\frac{1}{4\gamma}\norm{\hat b-b_\delta-\alpha_\delta \hat w}_Y^2
        +\gamma\gamma_\delta\norm{x_\delta-\hat x}_X^2.
        \end{aligned}
    \end{equation}
    Thus, again using Young's inequality and $\norm{\hat b - b_\delta} \le \delta$ from \cref{ass:linear:main}, we obtain
    \begin{equation}
        \label{eq:linear:subreg:final-est}
        \norm{x_\delta-\hat x}_X^2
        \le
        \frac{e_\delta}{\gamma\gamma_\delta}
        +\frac{1}{4\gamma^2\gamma_\delta}\norm{\hat b-b_\delta-\alpha_\delta \hat w}_Y^2
        \le
        \frac{e_\delta}{\gamma\gamma_\delta}
        +\frac{\delta^2}{2\gamma^2\gamma_\delta}
        +\frac{\alpha_\delta^2}{2\gamma^2\gamma_\delta}\norm{\hat w}_Y^2.
    \end{equation}
    This is the claim.
\end{proof}

If $\gamma_\delta \propto \alpha_\delta$, the following corollary shows norm convergence under similar parameter choices as in \cref{thm:linear:bregman}.

\begin{corollary}
    \label{cor:linear:strong-subreg}
    Suppose \cref{ass:linear:main} and the \emph{strong} source condition of \cref{ass:linear:strong-source-condition} hold at $\hat x$.
    If
    \begin{equation}
        \label{eq:linear:strong-subreg:convergence-cond}
        \lim_{\delta \downto 0} \frac{1}{\min\{\alpha_\delta, \gamma_\delta\}}(\alpha_\delta^2, \delta^2, e_\delta)=0,
    \end{equation}
    then
    \[
        \lim_{\delta \downto 0} \norm{x_\delta-\hat x}_X=0.
    \]
\end{corollary}

\subsection{Estimates based on a semi-strong source condition}

We now replace the assumption of strong metric subregularity, i.e., strong local subdifferentiability, with mere $(A, \gamma_\delta)$-strong local subdifferentiability.

\begin{assumption}[Semi-strong source condition]
    \label{ass:linear:semi-strong-source-condition}
    We assume that $\hat x \in X$ satisfies for some $\hat w \in Y$ the basic source condition  \eqref{eq:linear:basic-source-condition}.
    Moreover, for all $\delta>0$, for given $\alpha_\delta,\gamma_\delta>0$, assume that $J_\delta+\alpha_\delta R$ is $(A,\gamma_\delta)$-strongly locally subdifferentiable at $\hat x$ for $J_\delta'(\hat x)-\alpha_\delta A^*\hat w$ with respect to $\hat X$.
    The factor $\gamma>0$ and neighbourhood $U=U^{\hat x}$ of $(A,\gamma_\delta)$-strong local subdifferentiability must be independent of $\delta>0$.
    Then we say that $\hat x$ satisfies for $\hat w$ the \term{semi-strong source condition}.
\end{assumption}

It is not that easy to very the semi-strong source condition with $U^{\hat x} \supset U_\rho$ as in \cref{thm:linear:strong-subreg} without assumptions that would verify strong source condition of \cref{ass:linear:strong-source-condition}. We therefore drop this assumption for the next lemma at the cost of weaker results in the ensuing theorem.

\begin{lemma}
    \label{lemma:linear:subreg}
    Suppose \cref{ass:linear:main} and the \emph{semi-strong} source condition of \cref{ass:linear:semi-strong-source-condition} hold at some $\hat x \in X$ for some $\hat w$.
    If $x_\delta \in U^{\hat x}$, then
    \[
        \dist^2(x_\delta, \hat X)
        \le
        \frac{e_\delta}{\gamma\gamma_\delta}
        +\frac{\delta^2}{2\gamma^2\gamma_\delta}
        +\frac{\alpha_\delta^2}{2\gamma^2\gamma_\delta} \norm{\hat w}_Y^2.
    \]
\end{lemma}

\begin{proof}
    As in \eqref{eq:linear:subreg:base-est}, using the assumed $(A,\gamma_\delta)$-strong local subdifferentiability of $J_\delta+\alpha_\delta R$, we estimate
    \begin{equation*}
        \begin{aligned}[t]
        e_\delta
        &
        \ge
        [J_\delta+\alpha_\delta R](x_\delta)-[J_\delta+\alpha_\delta R](\hat x)
        \\
        &
        \ge
        \dualprod{J_\delta'(\hat x) - \alpha_\delta  A^*\hat w}{x_\delta-\hat x}
        + \gamma \norm{A(x_\delta-\hat x)}_Y^2
        + \gamma\gamma_\delta \dist^2(x_\delta, \hat X)
        \\
        &
        = \iprod{\hat b-b_\delta-\alpha_\delta \hat w}{A(x_\delta-\hat x)}
        + \gamma \norm{A(x_\delta-\hat x)}_Y^2
        + \gamma\gamma_\delta \dist^2(x_\delta, \hat X)
        \\
        &
        \ge
        -\frac{1}{4\gamma}\norm{\hat b-b_\delta-\alpha_\delta \hat w}_Y^2
        +\gamma \dist^2(x_\delta, \hat X).
        \end{aligned}
    \end{equation*}
    Now estimating as in \eqref{eq:linear:subreg:final-est} yields the claim.
\end{proof}

\begin{theorem}
    \label{thm:linear:subreg:weak}
    Suppose \cref{ass:linear:main} holds and that there exists a collection $\tilde X \subset \hat X$ of points satisfying the \emph{semi-strong} source condition of \cref{ass:linear:semi-strong-source-condition} with $\Union_{\hat x \in \tilde X} U^{\hat x} \supset \hat X$.
    Also suppose that $U_\rho$ defined in \eqref{eq:linear:strong-subreg:u-cond} is weakly or weakly-$*$ compact for some $\rho>0$, and each $U^{\hat x}$ for all $\hat x \in \tilde X$ correspondingly weakly or weakly-$*$ open.
    If
    \begin{equation}
        \label{eq:linear:subreg:weak:param-assumption}
        \lim_{\delta \downto 0} \frac{1}{\min\{\alpha_\delta, \gamma_\delta\}}(\alpha_\delta^2, \delta^2, e_\delta)=0,
    \end{equation}
    then
    \[
        \lim_{\delta \downto 0} \dist(x_\delta, \hat X)=0.
    \]
\end{theorem}

\begin{proof}
    Suppose, to reach a contradiction, for a sequence $\delta_k \downto 0$ that $\inf_k \dist(x_{\delta_k}, \hat X) > 0$.
    By \cref{lemma:linear:a-convergence} and \eqref{eq:linear:subreg:weak:param-assumption}, we have $x_{\delta_k} \in U_\rho$ for large enough $k$.
    Since $U_\rho$ is weakly(-$*$) compact, we can extract a subsequence, unrelabelled, such that also $x_{\delta_k} \weakto \tilde x$ weakly(-$*$).
    Since  $A\tilde x=\hat b$ for $\hat x \in \hat X$, we have using \cref{ass:linear:main} and  \eqref{eq:linear:subreg:weak:param-assumption} that
    \[
        \begin{aligned}[t]
        \lim_{k \to \infty} \norm{A x_{\delta_k}-b_{\delta_k}}^2
        &
        \le
        \lim_{k \to \infty}\left( \norm{A x_{\delta_k}-b_{\delta_k}}^2 + \alpha_k R(x_{\delta_k}) \right)
        \\
        &
        \le
        \lim_{k \to \infty}\left( \norm{A \hat x-b_{\delta_k}}^2 + \alpha_k R(\hat x) + e_{\delta_k} \right)
        \\
        &
        =
        \lim_{k \to \infty} \norm{\hat b-b_{\delta_k}}^2
        =
        \lim_{k \to \infty} \delta_k^2
        = 0
        \quad (\hat x \in \hat X).
        \end{aligned}
    \]
    Using \eqref{eq:linear:noise} and \eqref{eq:linear:subreg:weak:param-assumption}  yields $Ax_{\delta_k} \to \hat b$. Moreover, \eqref{eq:linear:noise} and \eqref{eq:linear:accuracy} in \cref{ass:linear:main} give
    \[
        R(x_{\delta_k}) \le \frac{\delta_k^2}{\alpha_{\delta_k}} + R(\hat x).
    \]
    Since $R$ is assumed convex and lower semicontinuous, it is weakly(-$*$) lower semicontinuous; see, e.g., \cite[Corollary 2.2]{ekeland1999convex} or \cite[Lemma 1.10]{clasonvalkonen2020nonsmooth}.
    Hence, using \eqref{eq:linear:subreg:weak:param-assumption}, we obtain $R(\tilde x) \le R(\hat x)$ and consequently $\tilde x \in \hat X$.
    By assumption, $\tilde x \in U^{\hat x}$ for some $\hat x \in \tilde X$ satisfying the semi-strong source condition of \cref{ass:linear:semi-strong-source-condition}.
    By the weak(-$*$) openness of $U^{\hat x}$, it follows that $x_{\delta_k} \in U^{\hat x}$ for $k$ large enough.
    Following the proof of \cref{lemma:linear:subreg} and using \eqref{eq:linear:subreg:weak:param-assumption}, we deduce that $\dist(x_{\delta_k}, \hat X) \to 0$.
\end{proof}

The following will be useful for verifying the strong metric subregularity or $(A, \gamma_\delta)$-strong local subdifferentiability required in \cref{ass:linear:strong-source-condition,ass:linear:semi-strong-source-condition}. The proof of the first lemma follows from that of the second by taking $\hat X=\{\hat x\}$ and expanding $\norm{\freevar}_\delta$.

\begin{lemma}
    \label{lemma:linear:to-prove:strong}
    Suppose $A$ and $R$ are as in \cref{ass:linear:main}.
    Then $J_\delta+\alpha_\delta R$ is strongly locally subdifferentiable at $\hat x$ for $J_\delta'(\hat x)-\alpha_\delta \hat d$ with respect to the norm $\norm{\freevar}_\delta$ if, for corresponding neighbourhood $U \ni \hat x$ and $\gamma>0$,
    \begin{equation*}
        \alpha_\delta[R(x)-R(\hat x)-\dualprod{\hat d}{x-\hat x}]
        +\left(\frac{1}{2}-\gamma\right)\norm{A(x-\hat x)}_2^2
        \ge
        \gamma\gamma_\delta \norm{x-\hat x}^2
        \quad (x \in U).
    \end{equation*}
\end{lemma}

\begin{lemma}
    \label{lemma:linear:to-prove:semistrong}
    Suppose $A$ and $R$ are as in \cref{ass:linear:main}.
    Then $J_\delta+\alpha_\delta R$ is $(A, \gamma_\delta)$-strongly locally subdifferentiable at $\hat x$ for $J_\delta'(\hat x)-\alpha_\delta \hat d$ with respect to $\hat X$ if, for corresponding neighbourhood $U \ni \hat x$ and $\gamma>0$,
    \begin{equation}
        \label{eq:linear:to-prove}
        \alpha_\delta[R(x)-R(\hat x)-\dualprod{\hat d}{x-\hat x}]
        +\left(\frac{1}{2}-\gamma\right)\norm{A(x-\hat x)}_2^2
        \ge
        \gamma\gamma_\delta \dist^2(x, \hat X)
        \quad (x \in U).
    \end{equation}
\end{lemma}

\begin{proof}
    Expanding \eqref{eq:subreg:strong-subdiff}, we need to prove
    \begin{multline*}
        \frac{1}{2}\norm{Ax-b_\delta}_2^2+\alpha_\delta R(x)
        -\frac{1}{2}\norm{A\hat x-b_\delta}_2^2-\alpha_\delta R(\hat x)
        \\
        \ge
        \dualprod{A^*(A\hat x-b_\delta)+\alpha_\delta\hat d}{x-\hat x} + \gamma \norm{A(x-\hat x)}^2 + \gamma \gamma_\delta \dist^2(x, \hat X)
        \quad (x \in U).
    \end{multline*}
    Using the properties of the Hilbert space norm (Pythagoras' identity), this rearranges as \eqref{eq:linear:to-prove}.
\end{proof}

\section{Examples}
\label{sec:examples}

We now look at a few examples that demonstrate \cref{thm:linear:strong-subreg,cor:linear:strong-subreg}.

\subsection{Basic examples}

Taking $R(x)=\frac{1}{2}\norm{x}_X^2$ and $e_\delta=0$, we recover from \cref{cor:linear:strong-subreg} classical results on Tikhonov regularisation \cite{engl2000regularization}:

\begin{theorem}[Norm-squared regularisation]
    \label{ex:linear:tikhonov}
    Let $X$ and $Y$ be Hilbert spaces, $A \in \linear(X; Y)$.
    Suppose $A \hat x=\hat b$ and $\hat x \in \range A^*$, i.e. $\hat x=-A^*\hat w$ for some $\hat w$.
    For all $\delta>0$, let $x_\delta \in X$ solve
    \[
        (A^*A + \alpha_\delta)x_\delta + A^*b_\delta = 0.
    \]
    Then
    $
        \norm{x_\delta-\hat x}_X^2
        \le
        \frac{\delta^2}{2\alpha_\delta}
        +\frac{\alpha_\delta}{2}\norm{\hat w}_Y^2.
    $
    In particular $\norm{x_\delta-\hat x} \downto 0$ as $\delta \downto 0$ provided we choose the regularisation parameter $\alpha_\delta \downto 0$ such that $\delta^2/\alpha_\delta \downto 0$.
 \end{theorem}

\begin{proof}
    Since $\subdiff R(x)=\{x\}$, the basic source condition \eqref{eq:linear:basic-source-condition} holds due to $\hat x \in \range A^*$.
    Since $R$ is strongly convex with parameter $\gamma=1$, the strong strong local subdifferentiability required by \cref{ass:linear:strong-source-condition} holds with $\gamma_\delta=\alpha_\delta$, $\gamma=1$, and $U=X$.
    Consequently, taking accuracy $e_\delta \equiv 0$, \cref{thm:linear:strong-subreg,cor:linear:strong-subreg} yield the claims.
\end{proof}

We can also add constraints. Denoting by $N_{[0, \infty)^\Omega}(\hat x)$ the normal cone to $[0, \infty)^\Omega$ in $L^2(\Omega)$, in the next result we take $R(x)=\frac{1}{2}\norm{x}_X^2 + \delta_{[0, \infty)^\Omega}(x)$ in $X=L^2(\Omega)$.

\begin{theorem}[Norm-squared regularisation with non-negativity constraints]
    Let $X=L^2(\Omega)$ for some $\Omega \subset \R^d$, and let $Y$ be a Hilbert space, $A \in \linear(X; Y)$.
    Suppose $A \hat x=\hat b$ and $\hat x$ satisfies the source condition $\hat x \in \range A^* - N_{[0, \infty)^\Omega}(\hat x)$, i.e. $\hat x \in - A^*\hat w - N_{[0, \infty)^\Omega}(\hat x)$ for some $\hat w$.
    For all $\delta>0$, let
    \[
        x_\delta \in \argmin_{0 \le x \in L^2(\Omega)} \frac{1}{2}\norm{Ax-b_\delta}_Y^2 + \frac{\alpha_\delta}{2}\norm{x}_X^2.
    \]
    Then
    $
        \norm{x_\delta-\hat x}_X^2
        \le
        \frac{\delta^2}{2\alpha_\delta}
        +\frac{\alpha_\delta}{2}\norm{\hat w}_Y^2
    $
    In particular $\norm{x_\delta-\hat x} \downto 0$ as $\delta \downto 0$ provided we choose the regularisation parameter $\alpha_\delta \downto 0$ such that $\delta^2/\alpha_\delta \downto 0$.
\end{theorem}

\begin{proof}
    Since $\subdiff R(x)=x+N_{[0, \infty)^\Omega}(\hat x)$, the basic source condition \eqref{eq:linear:basic-source-condition} holds due to $\hat x \in - A^*\hat w - N_{[0, \infty)^\Omega}(\hat x)$.
    Since $R$ is strongly convex with parameter $\gamma=1$, the strong strong local subdifferentiability required by \cref{ass:linear:strong-source-condition} holds with $\gamma_\delta=\alpha_\delta$, $\gamma=1$, and $U=X$.
    Consequently, taking accuracy $e_\delta \equiv 0$, \cref{thm:linear:strong-subreg,cor:linear:strong-subreg} yield the claims.
\end{proof}

We next look at nonsmooth regularisation, first in finite dimensions and then in infinite dimensions.

\subsection{$\ell^1$-regularised regression}

We now take $R(x)=\norm{x}_1$ in $\R^n$ with $A \in \R^{m \times n}$. The basic source condition \eqref{eq:linear:basic-source-condition} then holds if there exist $\hat x, \hat d \in \R^n$ such that
\begin{equation}
    \label{eq:lasso:basic-source-condition}
    A\hat x=\hat b
    \quad\text{and}\quad
    \hat d \in \range A^* \isect \Sign \hat x,
\end{equation}
where we recall for $x \in \R^n$, with $\Pi$ denoting a cartesian product of sets, that
\[
    \subdiff \norm{\freevar}_1(x)
    =
    \Sign x
    \defeq
    \prod_{k=1}^n \begin{cases}
        \{-1\}, & x_k<0, \\
        \{1\}, & x_k>0, \\
        [-1, 1], & x_k=0.
    \end{cases}
\]
We also write $\sign x$ for the (scalar) sign of $x \ne 0$.
As we recall, \eqref{eq:lasso:basic-source-condition} just means that $\hat x$ solves $\min_{Ax=\hat b} \norm{x}_1$. Then $\hat X$ is the set of these minimisers.
Otherwise said, \eqref{eq:lasso:basic-source-condition} holds if $\hat x$ is an admissible ground-truth ($A\hat x=\hat b$) and $\range A^* \isect \Sign \hat x \ne \emptyset$.

With $\hat d \in \Sign \hat x$, we write
\[
    Z(\hat x, \hat d) \defeq \{k \in \{1,\ldots,n\} \mid \hat x_k=0,\, \abs{\hat d_k} < 1\}
\]
and call $(\hat x, \hat d)$ \term{strictly complementary} if $Z(\hat x, \hat d)=\{k \in \{1,\ldots,n\} \mid \hat x_k=0\}$.
That is, $(\hat x, \hat d)$ are strictly complementary if $\hat x_k=0$ implies that $\hat d_k$ is not on the boundary of $[-1, 1]$.

The next result is extracted from \cite[proof of Theorem 2]{zhang2015necessary}.

\begin{lemma}
    \label{lemma:lasso:construction}
    Suppose $\hat x$ and $\hat d$ satisfy the source condition \eqref{eq:lasso:basic-source-condition}.
    Then there exists $\opt d \in \R^n$ such that $(\optx, \optd)$ is strictly complementary and satisfies \eqref{eq:lasso:basic-source-condition}.
\end{lemma}

\begin{proof}
    Let $r \in \R^n$ be defined by
    \[
        r_k \defeq \begin{cases}
            -1, & \hat x_k=0,\, \hat d_k=1,\\
            1, & \hat x_k=0,\, \hat d_k=-1,\\
            0, & \text{otherwise}.
        \end{cases}
    \]
    We try to replace $(\hat x, \hat d)$ by a strictly complementary solution by considering for $\alpha>0$, $t \defeq \norm{\hat x}_1$ and $\B_1 \subset \R^n$ the $1$-norm unit ball the problem
    \[
        \min_{x \in \R^n} J(x)=\left(\alpha \iprod{r}{x} + \delta_{t \B_1}(x)\right) + \delta_{\{\hat b\}}(Ax).
    \]
    Clearly $x=\hat x$ is a solution to this problem and $J(\hat x)=0$.
    Moreover $A\hat x = \hat b \in \ri \{\hat b\}=\{\hat b\}$, where $\ri$ denotes the relative interior.
    Therefore, the Fenchel--Rockafellar dual problem (see, e.g., \cite{ekeland1999convex,clasonvalkonen2020nonsmooth}) is
    \[
        \min_{w \in \R^m} Q(w) \defeq t\norm{A^*w+\alpha r}_\infty + \iprod{\hat b}{w},
    \]
    and a minimiser $\opt w$ satisfies $Q(\opt w)=-J(\hat x)=0$. In other words $t\norm{A^*\opt w+\alpha r}_\infty = - \iprod{\hat b}{\opt w}$.

    If $A^*\opt w=-\alpha r$, then for small enough $\alpha>0$ we have
    \[
        \opt d \defeq \hat d-A^*\opt w = \hat d+\alpha r \in \Sign \hat x \isect \range A^*
    \]
    and $(\hat x, \opt d)$ is strictly complementary. Thus it fulfills our claims.

    If $A^*\opt w\ne -\alpha r$, let $\alt w \defeq -\opt w/s$ for $s \defeq \norm{A^*\opt w+\alpha r}_\infty=-\iprod{\hat b}{\opt w}/t$.
    Then
    \begin{equation}
        \label{eq:lasso:perturb-ok}
        \iprod{\hat x}{A^*\alt w}
        =\iprod{A\hat x}{\alt w}
        =\iprod{\hat b}{\alt w}
        =-\inv s \iprod{\hat b}{\opt w}
        =t=\norm{\hat x}_1
    \end{equation}
    and
    \begin{equation}
        \label{eq:lasso:perturb-bound}
        \norm{A^*\alt w-\frac{\alpha}{s}r}_\infty \le 1.
    \end{equation}
    Let $k \in \{1,\ldots,n\}$.
    If $\hat x_k \ne 0$ or $k \in Z(\hat x, \hat d)$, \eqref{eq:lasso:perturb-bound} implies by the definition of $r_k$ that $[A^*\hat w]_k \in [-1, 1]$.
    Due to \eqref{eq:lasso:perturb-ok}, we must therefore have that $[A^*\hat w]_k = \hat d_k = \sign \hat x_k$ when $\hat x_k \ne 0$.
    If, on the other hand $\hat x_k=0$ and $k \not\in Z(\hat x, \hat d)$, then \eqref{eq:lasso:perturb-bound} shows that
    \[
        -1-\frac{\alpha}{s}\sign \hat d_k \le A^*\alt w \le 1-\frac{\alpha}{s}\sign \hat d_k.
    \]
    If $\hat d_k>0$, this guarantees $-2 < [A^*\alt w]_k < 1$.
    If $\hat d_k<0$, this guarantess $-1 < [A^*\alt w]_k < 2$.
    Consequently, for small enough $\alpha>0$,
    \[
        \opt d \defeq \frac{1}{2}(\hat d + A^*\alt w) \in \Sign \hat x \isect \range A^*,
    \]
    and $(\hat x, \opt d)$ is strictly complementary.
    Indeed, if $\hat d_k \in (-1, 1)$, then still $\opt d_k \in (-1, 1)$ due to $[A^*\alt w]_k \in [-1, 1]$.
    On the other hand, if $\hat d_k \in \{-1,1\}$ and $\hat x \ne 0$, we have just proved that $\hat d_k + [A^*\alt w]_k \in (-1, 1)$.
    Finally, if $\hat x_k \ne 0$, we have proved above that $[A^*\hat w]_k = \hat d_k \in \sign \hat x_k$.
    Thus $(\hat x, \opt d)$ it fulfills our claims.
\end{proof}

The next lemma shows that local subdifferentiability is requisite for the basic source condition \eqref{eq:lasso:basic-source-condition} to hold at a ground-truth $\hat x$ admitting a strictly complementary dual variable $\hat d$..

\begin{lemma}
    \label{lemma:lasso:sc}
    Suppose $(\hat x, \hat d)$ is strictly complementary and satisfies the basic source condition  \eqref{eq:lasso:basic-source-condition}.
    Let $F_\delta(x) \defeq \frac{1}{2}\norm{Ax-b_\delta}_2^2 + \alpha_\delta\norm{x}_1$ on $\R^n$ for some $\alpha_\delta \in (0, 1/2)$ and $b_\delta \in \R^m$.
    Then $\subdiff F_\delta$ is $(A, \alpha_\delta)$-strongly locally subdifferentiable at $\hat x$ for $\hat x^* \defeq A^*(A\hat x - b_\delta) + \alpha_\delta \hat d \in \subdiff F_\delta(\hat x)$ in some open neighbourhood $U=U^{\hat x}$ of $\hat x$. The factor $\gamma$ of $(A, \alpha_\delta)$-strong local subdifferentiability is independent of $\delta$, as is $U$.
\end{lemma}

\begin{proof}
    By \cref{lemma:linear:to-prove:semistrong} we need to prove
    \begin{equation}
        \label{eq:lasso:to-prove}
        \alpha_\delta[\norm{x}_1-\norm{\hat x}_1-\iprod{\hat d}{x-\hat x}]
        +\left(\frac{1}{2}-\gamma\right)\norm{A(x-\hat x)}_2^2
        \ge
        \gamma\gamma_\delta \dist^2(x, \hat X)
        \quad (x \in U).
    \end{equation}
    Let
    \begin{equation}
        \label{eq:lasso:m}
        M \defeq  A^*A + \sum_{k \in Z(\hat x, \hat d)} \mathbb{1}_k \mathbb{1}_k^\top,
    \end{equation}
    where $\mathbb{1} \defeq (1, \ldots, 1)$.
    Suppose $z \in \ker M$.
    Then $z \in \ker A$ and $z_k=0$ for all $k \in Z(\hat x, \hat d)$.
    We will show that also $\sum_{k \not \in Z(\hat x, \hat d)} z_k=0$ and $\hat x + z \in \hat X$ if $\norm{z}$ is small enough.
    Indeed, suppose, to reach a contradiction, that this does not hold.
    Let $\epsilon>0$.
    We may assume that $\norm{z} \le \epsilon$.
    We have $A(\hat x+z)=A \hat x=\hat b$ and, for small enough $\epsilon>0$,
    \[
        \norm{\hat x+z}_1
        =
        \sum_{k \not\in Z(\hat x, \hat d)} \abs{\hat x_k+z_k}
        =
        \sum_{k: \hat x_k \ne 0} \abs{\hat x_k+z_k}
        =
        \sum_{k: \hat x_k \ne 0}\left( \abs{\hat x_k}+z_k\right)
        =\norm{\hat x}_1 + \sum_{k \not \in Z(\hat x, \hat d)} z_k.
    \]
    If now $\sum_{k \not \in Z(\hat x, \hat d)} z_k<0$, this contradicts $\hat x$ being an $\norm{\freevar}_1$-minimising solution to $A\hat x=\hat b$. Since also $-z \in \ker A$, we must therefore have $\sum_{k \not \in Z(\hat x, \hat d)} z_k=0$.
    Moreover, $\hat x + z \in \hat X$.

    For now, let $U=U^{\hat x}$ be an arbitrary bounded neighbourhood of $\hat x$ and $\rho=\sup_{x \in U} \norm{x_k}_\infty$.
    Then with $\beta_0 \defeq \inv\rho \inf_{k \in Z(\hat x, \hat d)}(1-\abs{\hat d_k})>0$ we have
    \begin{equation}
        \label{eq:subreg:l1-semi-strong-subreg}
        \begin{aligned}[t]
        \norm{x}_1 - \norm{\hat x}_1 - \iprod{\hat d}{x-\hat x}
        &
        \ge
        \sum_{k \in Z(\hat x, \hat d)}(\abs{x}-\abs{\hat x} - \hat d_k(x_k-\hat x_k))
        =
        \sum_{k \in Z(\hat x, \hat d)}(\abs{x} - \hat d_k(x_k))
        \\
        &
        \ge
        \sum_{k \in Z(\hat x, \hat d)} (1-\abs{\hat d})\abs{x_k}
        \ge
        \rho\beta_0 \sum_{k \in Z(\hat x, \hat d)} \abs{x_k}
        \ge
        \beta_0 \sum_{k \in Z(\hat x, \hat d)} \abs{x_k}^2
        \quad (x \in U).
        \end{aligned}
    \end{equation}
    Since $\alpha_\delta \in (0, 1/2)$, we may find $\gamma \in (0, 1/2)$ with $\alpha_\delta  \le \frac12 - \gamma$.
    By \cref{eq:subreg:l1-semi-strong-subreg}, for some $\beta \defeq \min\{1,\beta_0\}$ we have
    \begin{multline}
        \label{eq:lasso:est1}
        \alpha_\delta[\norm{x}_1-\norm{\hat x}_1-\iprod{\hat d}{x-\hat x}]
        +\left(\frac{1}{2}-\gamma\right)\norm{A(x-\hat x)}_2^2
        \\
        \ge
        \alpha_\delta \beta\sum_{k \in Z(\hat x, \hat d)} \abs{x_k}^2
        +\left(\frac{1}{2}-\gamma\right)\norm{A(x-\hat x)}_2^2
        \ge
        \alpha_\delta \beta \iprod{M(x-\hat x)}{x-\hat x}.
    \end{multline}
    With $x \in U$, write $x-\hat x=z+w$ where $z \in \ker M$ and $w \perp \ker M$.
    That is, $x-\opt x=w$ for $\opt x=\hat x+z$. By the previous paragraph, $\opt x \in \hat X$ if $\norm{z}$ is small enough, i.e., the neighbourhood $U$ of $\hat x$ is small enough and $x \in U$.
    Thus, for $\lambda_{\min}>0$ the minimal non-zero eigenvalue of $M$,
    \[
         \alpha_\delta \beta \iprod{M(x-\hat x)}{x-\hat x}
         =
         \alpha_\delta \beta \iprod{Mw}{w}
         \ge
         \alpha_\delta \beta \lambda_{\min} \norm{w}_X^2
         =
         \alpha_\delta \beta \lambda_{\min} \norm{x-\opt x}_X^2.
    \]
    This with \eqref{eq:lasso:est1} proves \eqref{eq:lasso:to-prove}, hence the claim for $\gamma_\delta=\alpha_\delta$ and any $0 < \gamma < \min\{1/2,\beta \lambda_{\min}\}$.
\end{proof}

Finally, we may show that asymptotically accurate and sufficiently regularised solutions to $\ell^1$-regularised regression otherwise unconditionally converge to the \emph{set} of $1$-norm-minimising ground-truths.
Based on the practical implementation of the procedure, $\ell^1$-regularised regression is also known in the literature as “Lasso” for “least absolute shrinkage and selection operator”.

\begin{theorem}[Lasso]
    \label{thm:lasso:main}
    Let $R(x)=\norm{x}_1$ and and $A \in \R^{m \times n}$  with $X=\R^n$ and $Y=\R^m$.
    Suppose \cref{ass:linear:main} holds and that the accuracy $e^\delta$ and regularisation parameter $\alpha_\delta>0$, moreover, satisfy
    \[
        \lim_{\delta \downto 0} \left(\alpha_\delta, \frac{\delta^2}{\alpha_\delta}, \frac{e_\delta}{\alpha_\delta}\right)=0.
    \]
    Then $\dist(x_\delta, \hat X) \to 0$ as $\delta \downto 0$.
\end{theorem}

\begin{proof}
    We verify the conditions \cref{thm:linear:subreg:weak}.
    The function $J(x) \defeq \delta_{\{x \mid Ax=\hat b\}}(x) + \norm{x}_1$ is coercive, lower semicontinuous, and bounded from below.
    Therefore the problem $\min J$ has a solution $\hat x \in \hat X$. As we have already discussed, these  solutions are characterised by the source condition \eqref{eq:lasso:basic-source-condition} for some $\hat d$.
    Thus by \cref{lemma:lasso:construction}, for every $\hat x \in \hat X$ there exists $\hat d$ such that $(\hat x, \hat d)$ is strictly complementary and satisfies \eqref{eq:lasso:basic-source-condition}.
    \Cref{lemma:lasso:sc} shows for all $\hat x \in \hat X$ and all $\delta>0$ such that $\alpha_\delta \in (0, 1/2)$ that $\subdiff[J_\delta+\alpha_\delta R]$ is $(A, \alpha_\delta)$-strongly locally subdifferentiable at $\hat x$ for $\hat x^*_\delta \defeq A^*(A\hat x-b_\delta)+\alpha_\delta \hat d$ with the parameter $\gamma>0$ independent of $\delta>0$.
    Consequently the semi-strong source condition of \cref{ass:linear:semi-strong-source-condition} holds for all $\hat x \in \hat X$ in some neighbourhood $U^{\hat x} \ni \hat x$.
    Clearly $\Union_{\hat x \in \hat X} U^{\hat x} \supset \hat X$.
    Moreover, $U_\rho$ is compact by the lower semicontinuity and coercivity of $R$ and the continuity of $A$, and $U^{\hat x}$ is open.
    By the finite-dimensionality of $X$ the respective weak compactness and weak openness assumptions of \cref{thm:linear:subreg:weak} hold.
    The rest follows from \cref{thm:linear:subreg:weak} and $\gamma_\delta=\alpha_\delta$.
\end{proof}

\begin{remark}
    If $M$ defined in \eqref{eq:lasso:m} is positive definite, then the strong source condition of \cref{ass:linear:strong-source-condition} holds at $\hat x$. In that case we may apply \cref{thm:linear:strong-subreg} to obtain $\norm{x_\delta-\hat x} \to 0$ and the estimates \eqref{eq:linear:strong-subreg:estimate}.
\end{remark}

\subsection{Total variation regularised image reconstruction}

Suppose, as before, that $\Omega \subset \R^n$.
On the space $X=L^2(\Omega)$, define the \term{total variation} regulariser
\[
    R(x)
    =
    \begin{cases}
    \norm{Dx}_{\Meas(\Omega)}, & \BVspace(\Omega) \isect L^2(\Omega), \\
    \infty, & \text{otherwise}.
    \end{cases}
\]
Here $D \in \linear(\BVspace(\Omega); \Meas(\Omega))$ is the distributional differential, mapping functions of bounded variation on a domain $\Omega \subset \R^m$ to Radon measures.
The Radon norm $\norm{Dx}_{\Meas(\Omega)} \defeq \int_\Omega \d\abs{Dx}$, where $\abs{Dx}$ is the total variation measure of $Dx$.
For details on functions of bounded variation we refer to \cite{ambrosio2000fbv}.
For brevity we write $\norm{\freevar}_2=\norm{\freevar}_{L^2(\Omega)}$ and $\norm{\freevar}_\Meas = \norm{\freevar}_{\Meas(\Omega; \R^m)}$.

Any $\hat d \in L^2(\Omega) \isect \subdiff\norm{D\freevar}_{\Meas(\Omega; \R^m)}(\hat x)$ satisfies by \cite{anzellotti1983pairings} or \cite[Proposition 5 \& Lemma 3 \& Definition 11]{meyer2002oscillating} for some $\hat \phi \in L^\infty(\Omega; \R^m)$ that
\begin{subequations}
\label{eq:tv-basic-source-cond}
\begin{equation}
    \label{eq:tv:subdiff}
    \hat d = -\divergence \hat \phi,
    \quad
    \norm{\hat\phi}_\infty \le 1,
    \quad
    \hat\phi D\hat x = \abs{D\hat x}.
\end{equation}
Hence the basic source condition \eqref{eq:linear:basic-source-condition} reduces to \eqref{eq:tv:subdiff} with
\begin{equation}
    A\hat x=\hat b
    \quad\text{and}\quad
    \hat d \in \range A^*.
\end{equation}
\end{subequations}
The set $\hat X$ is given by $\hat x$ satisfying \eqref{eq:tv-basic-source-cond} for some $\hat d$.

\begin{remark}
    It may often seem from the literature that \emph{all} $x^* \in \subdiff\norm{D\freevar}_{\Meas}(x)$ would have the form $x^*=-\divergence \phi$ from some $\phi \in L^\infty(\Omega; \R^m)$.
    This is not the case. Consider, for example, the step function $\hat x=\chi_{[0, \infty)}$ on $\R$.
    Then $D\hat x=\delta_{\{0\}}$. Take $\hat x^* \in \BVspace(\R)^*$ given by $\hat x^*(x)=D^jx(\R)$ for all $x \in \BVspace(\R)$, i.e., the measure of the jump part of the differential of $x$. Then $\hat x^*(\hat x)=\norm{D \hat x}_{\Meas}$ and $\norm{Dx}_{\Meas} \ge \hat x^*(x)$ for all $x \in \BVspace(\R)$. Thus
    \[
        \norm{Dx}_{\Meas}  - \norm{D \hat x}_{\Meas} \ge \hat x^*(x-\hat x)
    \]
    meaning that $\hat x^* \in \subdiff\norm{D\freevar}_{\Meas}(\hat x)$.
    However, if we had $\hat x^* = -\divergence \phi$, then we would have $\hat x^*(x)=Dx(\phi)$ contradicting the definition of $\hat x^*$ for suitable $x \in \BVspace(\R)$.
    It is thus important in our overall example that \emph{we are actually working in $L^2(\Omega)$}: $\range A^* \subset L^2(\Omega)$. This allows us to limit our attention to subderivatives $x^* \in L^2(\Omega)$.
\end{remark}

Let $\Lebesgue$ denote the Lebesgue measure on $\R^m$.
For a collection $\mathcal{O}$ of disjoint measurable $\Theta \subset \Omega$ with $\Lebesgue(\Theta)>0$, we define the centring operator $K_{\mathcal{O}}: L^2(\Omega) \to L^2(\Union_{\Theta \in \mathcal{O}} \Theta)$ by
\[
    (K_{\mathcal{O}} x)|\Theta \defeq x|\Theta - \frac{1}{\Lebesgue(\Theta)}\int_\Theta x \d \Lebesgue
    \quad
    (\Theta \in \mathcal{O}).
\]
Here $x|\Theta$ denotes the restriction of $x$ on the subdomain $\Theta$.
For brevity, we call the possibly empty collection $\mathcal{O}$ of disjoint subsets $\Theta \subset \Omega$ a \term{collection of flat areas} for $x$” if each $\Theta \in \mathcal{O}$ has a Lipschitz boundary and is such that $\abs{Dx}(\Theta)=0$ (i.e., $x$ is a.e.~constant in $\Theta$). If, moreover, \eqref{eq:tv:subdiff} holds, and $\sup_{\xi \in \Theta} \abs{\hat\phi(\xi)} < 1$ for all $\Theta \in \mathcal{O}$,
\begin{center}
     we call $\mathcal{O}$ a  \term{collection of strictly flat areas} for $\hat x$ at $\hat \phi$.
\end{center}
Note that $K_{\mathcal{O}} x=0$ if $\mathcal{O}$ is a collection of flat areas for $x$.

\begin{remark}[Stair-casing]
    Image reconstructions based on Tikhonov-style total variation regularisation commonly exhibit large flat areas or “stair-casing” \cite{ring2000structural,bredies2019sparsity}.
    The \emph{strictness} of the collection of flat areas $\mathcal{O}$ can be related to strict complementarity conditions in optimisation.
\end{remark}

We start with a technical lemma.

\begin{lemma}
    \label{lemma:tv:technical}
    Let $\hat x, \hat d \in L^2(\Omega)$ and $\hat \phi \in L^\infty(\Omega; \R^m)$ satisfy \eqref{eq:tv:subdiff}.
    Let $\mathcal{O}$ be a finite collection of strictly flat areas for $\hat x$ at $\hat \phi$ and suppose the neighbourhood $U \ni \hat x$ satisfies for some $C>0$ that
    \begin{equation}
        \label{eq:tv:neigh}
        U \subset \{x \in L^2(\Omega) \mid \norm{Dx}_{\Meas} \le C \}.
    \end{equation}
    Then, for some constant $\epsilon=\epsilon(\mathcal{O}, \hat d)>0$,
    \[
        \norm{Dx}_{\Meas}
        -\norm{D\hat x}_{\Meas}
        -
        \dualprod{\hat d}{x-\hat x}
        \ge
        \frac{\epsilon}{C} \norm{K_{\mathcal{O}}(x-\hat x)}_{L^2(\Union_{\Theta \in \mathcal{O}}\Theta)}^2
        \quad (x \in U).
    \]
\end{lemma}

\begin{proof}
    Observing \eqref{eq:tv:subdiff}, we have
    \begin{equation}
        \label{eq:tv:dx-est}
        \begin{aligned}[t]
        \norm{Dx}_{\Meas}
        -\norm{D\hat x}_{\Meas}
        -
        \dualprod{\hat d}{x-\hat x}
        &
        =
        \int_\Omega \d \abs{Dx}-\int_\Omega \d \abs{D\hat x}
        -\int_\Omega \hat\phi d(Dx-D\hat x)
        \\
        &
        \ge
        \int_\Omega 1-\abs{\hat \phi} \d\abs{Dx}.
        \end{aligned}
    \end{equation}
    We recall that Poincaré's inequality \cite[Remark 3.50]{ambrosio2000fbv} establishes for some constants $C_\Theta>0$ that
    \[
        \norm{K_{\{\Theta\}}x}_{L^2(\Theta)} \le C_\Theta \norm{Dx}_{\Meas(\Theta)}
        \quad (\Theta \in \mathcal{O},\, x \in L^2(\Omega)).
    \]
    Since, by assumption, $\abs{D\hat x}(\Theta)=0$ and $1-\abs{\hat \phi} \ge \epsilon_{\Theta,\hat \phi} > 0$ for some $\epsilon_{\Theta,\hat\phi}>0$ for every $\Theta \in \mathcal{O}$, we therefore obtain
    \begin{equation}
        \label{eq:tv:strict-est}
        \int_\Omega 1-\abs{\hat \phi} \d\abs{Dx}
        \ge
        \sum_{\Theta \in \mathcal{O}}
        \epsilon_{\Theta,\hat\phi} \int_\Theta \d\abs{Dx}
        \ge
        \sum_{\Theta \in \mathcal{O}}
        \inv C_\Theta \epsilon_{\Theta, \hat\phi} \norm{K_{\{\Theta\}} x}_{L^2(\Theta)}.
    \end{equation}
    (In particular, if $m=2$ and $\Theta=\B(\xi, \rho)$, we have $C_\Theta=C$ independent of $\xi$ and $\rho$.)
    On the other hand, due to \eqref{eq:tv:neigh}, Poincaré's inequality, and $K_{\{\Theta\}} \hat x =0$, we also have
    \[
        C
        \ge \norm{D x}_{\Meas(\Omega; \R^m)}
        \ge \norm{D x}_{\Meas(\Theta); \R^m)}
        \ge \inv C_\Theta\norm{K_{\{\Theta\}} x}_{L^2(\Theta)}
        = \inv C_\Theta\norm{K_{\{\Theta\}} (x-\hat x)}_{L^2(\Theta)}.
    \]
    Combining this estimate with \eqref{eq:tv:dx-est} and \eqref{eq:tv:strict-est} yields the claimed estimate with $\epsilon=\inf_{\Theta \in \mathcal{O}} \epsilon_{\Theta,\hat\phi}/C_\Theta^2$.
    Since we assumed $\mathcal{O}$ to be finite, we have $\epsilon>0$.
\end{proof}

\begin{remark}
    If the constant $C_\Theta$ from Poincaré's inequality is bounded over all $\Theta \in \mathcal{O}$, and we have $\sup_{\Theta \in \mathcal{O}\, \xi \in \Theta} \abs{\hat\phi(\xi)} < 1$, then the finiteness assumption on $\mathcal{O}$ can be dropped.
\end{remark}

A discretised version of the next lemma on total variation can be found in \cite[Appendix A]{jauhiainen2019gaussnewton}. It says that with regard to strong metric subregularity, the lack of positivity of $A^*A$ can be compensated for by the strictly flat areas.

\begin{lemma}
    \label{lemma:tv:strong-subreg}
    Let $F_\delta(x) \defeq \frac{1}{2}\norm{Ax-b_\delta}_{2}^2 + \alpha_\delta\norm{Dx}_{\Meas}$ on $X=L^2(\Omega)$ with $\Omega \subset \R^n$ and $\alpha_\delta \in (0, 1/2)$.
    Suppose:
    \begin{enumerate}[label=(\roman*)]
        \item $\hat x \in L^2(\Omega) \isect \BVspace(\Omega)$, and $\hat d \in L^2(\Omega)$, and $\hat \phi \in L^\infty(\Omega; \R^m)$ satisfy \eqref{eq:tv:subdiff}.
        \item There exists a a finite collection $\mathcal{O}$ of strictly flat areas for $\hat x$ such that
        \begin{equation}
            \label{eq:tv:ellipticity}
            K_{\mathcal{O}}^*K_{\mathcal{O}} + A^*A \ge \alt\epsilon\Id
            \text{ for some } \alt\epsilon>0.
        \end{equation}
        \item $U \subset L^2(\Omega)$ satisfies \eqref{eq:tv:neigh} for some $C>0$.
    \end{enumerate}
    Then $F$ is strongly locally subdifferentiable in $U$ at $\hat x$ for $\hat x^* \defeq A^*(A\hat x-b_\delta) + \alpha_\delta \hat d \in \subdiff F_\delta(\hat x)$  with respect to the norm $\norm{\freevar}_\delta$ with $\gamma_\delta=\alpha_\delta$. The factor $\kappa$ (or $\gamma$) of strong metric subregularity is independent of $\delta$.
\end{lemma}

\begin{proof}
    By \cref{lemma:linear:to-prove:strong} we need to prove
    \begin{equation}
        \label{eq:tv:main-est}
        L \defeq
        \alpha_\delta\left(
            \norm{Dx}_{\Meas}
            -\norm{D\hat x}_{\Meas}
            - \dualprod{\hat d}{x-\hat x}
        \right)
        +\left(\frac{1}{2}-\gamma\right)\norm{A(x-\hat x)}_2^2
        \ge
        \gamma\gamma_\delta\norm{x-\hat x}_X^2
        \quad (x \in U).
    \end{equation}
    \Cref{lemma:tv:technical} provides for some $\epsilon=\epsilon(\Theta, \hat d)$ the estimate
    \[
        \norm{Dx}_{\Meas}
        -\norm{D\hat x}_{\Meas}
        -\dualprod{\hat d}{x-\hat x}
        \ge
        \frac{\epsilon}{C} \norm{K_{\mathcal{O}}(x-\hat x)}_{L^2(\Theta)}^2.
    \]
    This and \eqref{eq:tv:ellipticity} yield
    \[
        L \ge
        \min\left\{\frac{\epsilon\alpha_\delta}{C}, \frac{1}{2}-\gamma\right\}\alt\epsilon\norm{x-\hat x}_2^2.
    \]
    Since we assume that $\alpha_\delta \in (0, 1/2)$, taking $\gamma_\delta=\alpha_\delta$ and $\alpha_\delta  \le \frac12 - \gamma$ with $\gamma \in (0, \frac12)$, we now prove \eqref{eq:tv:main-est} for small enough $\gamma>0$ independent of $\delta>0$.
\end{proof}

Provided the strictly flat areas in the ground-truth compensate for the kernel of the forward operator, we can now show the convergence of total variation regularised approximate solutions:

\begin{theorem}[Total variation regularised image reconstruction]
    \label{thm:tv:main}
    Let $R(x)=\norm{Dx}_{\Meas(\Omega; \R^m)}$ in $X=L^2(\Omega)$ and suppose \cref{ass:linear:main} holds.
    Also suppose $\hat x, \hat d \in L^2(\Omega)$ satisfy \eqref{eq:tv-basic-source-cond} and that there exists a corresponding collection $\mathcal{O}$ of strictly flat areas satisfying \eqref{eq:tv:ellipticity}.
    Also suppose that the accuracy and regularisation parameters satisfy
    \[
        \lim_{\delta \downto 0} \left(\alpha_\delta, \frac{\delta^2}{\alpha_\delta}, \frac{e_\delta}{\alpha_\delta}\right)=0.
    \]
    Then $\norm{x_\delta-\hat x}_2 \downto 0$.
    Moreover, \eqref{eq:linear:strong-subreg:estimate} holds for small enough $\delta>0$.
\end{theorem}

\begin{proof}
    For some $\rho>0$, let
    \[
        U \defeq U_\rho = \{x \in L^2(\Omega) \mid \norm{A(x-\hat x)} \le \rho,\, \norm{Dx}_{\Meas} \le \norm{D\hat x}_{\Meas} + \rho\}
    \]
    Then \eqref{eq:linear:strong-subreg:u-cond} holds as does \eqref{eq:tv:neigh} with $C \defeq \norm{D\hat x}_{\Meas} + \rho$.
    \Cref{lemma:tv:strong-subreg} now shows that $J_\delta+\alpha_\delta R$ is strongly locally subdifferentiable at $\hat x$ for $\hat x^*_\delta \defeq A^*(A\hat x-b_\delta)+\alpha_\delta \hat d$ with respect to the norm $\norm{\freevar}_\delta$ with $\gamma_\delta=\alpha_\delta$ and with the parameter $\gamma>0$ independent of $\delta>0$.
    Since \eqref{eq:tv-basic-source-cond} verifies the basic source condition \eqref{eq:linear:basic-source-condition}, this verifies the strong source condition of \cref{ass:linear:strong-source-condition}.
    Hence we may apply \cref{thm:linear:strong-subreg,cor:linear:strong-subreg} to deduce the claim.
\end{proof}

\section{Nonlinear inverse problems and general discrepancies}
\label{sec:nonlinear}

We now consider the nonlinear inverse problem $A(\hat x)= \hat b$, which for corrupted data $b_\delta$ we solve via
\begin{equation*}
    \min_{x \in X} J_\delta(x) + \alpha_\delta R(x)
    \quad\text{where}\quad
    J_\delta(x) \defeq E(A(x)-b_\delta)
\end{equation*}
for some convex, proper, lower semicontinuous regularisation functional $R: X \to \extR$ and a data fidelity $E: X \to \R$. We assume  the forward operator $A \in C^1(X; Y)$ on the Banach spaces $X$ and $Y$.
We will shortly impose assumptions on $E$.

\subsection{A basic source condition and general assumptions}

We again take as our starting point for admissible ground-truths $\hat x$ those that minimise $R$, i.e., solve the problem
\begin{equation}
    \label{eq:nonlinear:minimum-norm}
    \min_{A(x)=\hat b} R(x).
\end{equation}
We use the theory of Clarke subdifferentials on extended real-valued functions \cite[Section 2.9]{clarke1990optimization} to write the necessary optimality conditions for $\hat x$ to solve this problem.
Indeed, by the Fermat principle, $0 \in J(\hat x)$ for $J(x) \defeq R(x)+\delta_{\{x: A(x)=\hat b\}}$.
Since $A \in C^1(X; Y)$ is regular in the sense of Clarke's theory, as is $R$ as a convex function, the sum and composition rules \cite[Theorems 2.9.8 \& 2.9.9]{clarke1990optimization} hold as equalities for $J$. Thus the Fermat principle writes out as $\hat x$ having to satisfy for some $\hat w \in Y^*$ the \term{basic source condition}
\begin{equation}
    \label{eq:nonlinear:basic-source-condition}
    A'(\hat x)=\hat b
    \quad\text{and}\quad
    A'(\hat x)^*\hat w + \subdiff R(\hat x) \ni 0.
\end{equation}
We write $\hat X$ for the set of $\hat x$ satisfying \eqref{eq:nonlinear:basic-source-condition}.
This set may be larger than the set of minimisers of \cref{eq:nonlinear:minimum-norm}.

We now state our main assumption regarding accuracy and how $E$ relates to the noise parameter $\delta>0$.
Essentially, the magnitude of $E'$ has to be compatible with the parameterisation $\delta$ of the corruption level, and the convergence below zero of $E$ has to be Hölderian as the corruption in the data vanishes.

\begin{assumption}[Corruption level and solution accuracy]
    \label{ass:nonlinear:energy-noise}
    On Banach spaces $X$ and $Y$, $A \in C^1(X; Y)$, and $R: X \to \extR$ is convex, proper, and lower semicontinuous, and $E: Y \to \R$ is convex and Fréchet differentiable.
    For given accuracy parameters $e_\delta \ge 0$ and all $\hat x \in \hat X$ we have
    \begin{equation}
        \label{eq:nonlinear:accuracy}
        [J_\delta + \alpha_\delta R](x_\delta) - [J_\delta+\alpha_\delta R](\hat x) \le e_\delta
        \quad (\delta > 0).
    \end{equation}
    Moreover, the parametrisation $\delta>0$ of the noise or corruption level is such that
    \[
        \norm{E'(b_\delta-\hat b)}_{Y^*} \le \delta.
    \]
    and, for some $C',q>0$,
    \[
        E(\hat b-b_\delta) \le C'\delta^q
    \]
    Finally, for some $C, p>0$ the function $E$ satisfies the pseudo-Hölder estimate
    \begin{equation}
        \label{eq:nonlinear:e-triangle}
        \inv C E(z) \le E(w)+\norm{E'(z-w)}_Y^p
        \quad (z, w \in Y).
    \end{equation}
\end{assumption}

\begin{example}
    \label{ex:nonlinear:l2}
    Let $E(z) \defeq \frac{1}{2}\norm{z}_Y^2$ on a Hilbert space $Y$. Then $E'(z-w)=z-w$ so by \cref{ass:nonlinear:energy-noise} the noise level has to satisfy
    \[
        \norm{b_\delta-\hat b}_Y \le \delta
    \]
    and with this we can take $q=2$ and $C'=1/2$.
    Also \eqref{eq:nonlinear:e-triangle} holds with $C=3$ and $p=2$.
    Indeed
    \[
        \begin{aligned}
        \frac{1}{2C}\norm{z}_Y^2
        &
        = \frac{1}{2C}\norm{w-(z-w)}_Y^2
        \\
        &
        \le
        \left(\frac{1}{2C}+\frac{C-1}{2C}\right)\norm{z}_Y^2
        +\left(\frac{1}{2C}+\frac{C}{2(C-1)}\right)\norm{z-w}_Y^2
        \le \frac{1}{2}\norm{w}_Y^2 + \norm{z-w}_Y^2.
        \end{aligned}
    \]
\end{example}

\subsection{An estimate based on a strong source condition}

We modify the strong source condition of \cref{ass:linear:strong-source-condition} for nonlinear $A$.

\begin{assumption}[Strong source condition; nonlinear case]
    \label{ass:nonlinear:strong-source-condition}
    Assume that $\hat x \in X$ satisfies for some $\hat w \in Y$ the basic source condition  \eqref{eq:nonlinear:basic-source-condition}.
    Moreover, for all $\delta>0$, for given $\alpha_\delta,\gamma_\delta>0$, assume that $J_\delta+\alpha_\delta R$ is strongly locally subdifferentiable at $\hat x$ for $J_\delta'(\hat x)-\alpha_\delta A'(\hat x)^*\hat w$ with respect to the norm
    \[
        \norm{x}_\delta \defeq \sqrt{\norm{A'(\hat x)x}_Y^2 + \gamma_\delta \norm{x}_X^2} \quad (x \in X).
    \]
    The factor $\gamma>0$ of strong local subdifferentiability must be independent of $\delta>0$ and, for some $\rho>0$, we must have
    \begin{equation}
        \label{eq:nonlinear:strong-subreg:u-cond}
        U \supset U_\rho \defeq \{x \in X \mid \norm{A(x)-A(\hat x)} \le \rho,\, R(x) \le R(\hat x) + \rho\}.
    \end{equation}
    Then we say that $\hat x$ satisfies for $\hat w$ the \term{strong source condition}.
\end{assumption}

We again start with a simple bound:

\begin{lemma}
    \label{lemma:nonlinear:a-convergence}
    Suppose \cref{ass:nonlinear:energy-noise} holds.
    Then
    \[
        E(A(x_\delta)-A(\hat x)) \le C(e_\delta + C'\delta^q + \delta^p + \alpha_\delta R(\hat x))
    \]
    and
    \[
        R(x_\delta) \le R(\hat x) + \frac{e_\delta+C'\delta^q + \delta^p}{\alpha_\delta}.
    \]
\end{lemma}

\begin{proof}
    By \cref{ass:nonlinear:energy-noise} and \eqref{eq:nonlinear:accuracy}, since $A(\hat x)=\hat b$, we have
    \[
        \begin{aligned}
        \inv C E(A(x_\delta)-A(\hat x))
        + \alpha_\delta R(x_\delta)
        &
        \le
        E(A(x_\delta)-b_\delta)
        + \alpha_\delta R(x_\delta)
        + \norm{E'(b_\delta-\hat b)}_Y^p
        \\
        &
        \le
        E(\hat b-b_\delta) + \alpha_\delta R(\hat x)
        + e_\delta
        + \norm{E'(b_\delta-\hat b)}_Y^p
        \\
        &
        \le e_\delta + C'\delta^q + \delta^p + \alpha_\delta R(\hat x).
        \end{aligned}
    \]
    This finishes the proof.
\end{proof}

The next result generalises \cref{thm:linear:strong-subreg} to approximately linear operators $A$. The claim only differs by the factor $\eta$, which becomes the unit if $A$ is linear.

\begin{theorem}
    \label{thm:nonlinear:strong-subreg}
    Suppose \cref{ass:nonlinear:energy-noise} holds along with the strong source condition of \cref{ass:nonlinear:strong-source-condition} at some $\hat x$ for $\hat w$.
    Suppose $(e_\delta+\delta^q + \delta^p)/\alpha_\delta \downto 0$ and $\alpha_\delta \downto 0$ as $\delta \downto 0$.
    Then there exists $\delta>0$ such that if  $\delta \in (0, \bar\delta)$, we have
    \begin{equation}
        \label{eq:nonlinear:strong-subreg:estimate}
        \frac{1}{2}\norm{x_\delta-\hat x}_X^2
        \le
        \frac{e_\delta}{\gamma\gamma_\delta}
        +\frac{\delta^2}{2\gamma^2\gamma_\delta}
        +\frac{\alpha_\delta^2}{2\gamma^2\gamma_\delta}\norm{\hat w}_Y^2.
    \end{equation}
\end{theorem}

\begin{proof}
    Since $A(\hat x)=\hat b$, we have
    \[
        J_\delta'(\hat x) - \alpha_\delta A'(\hat x)^*\hat w
        =A'(\hat x)^*E'(A(\hat x)-b_\delta)-\alpha_\delta A'(\hat x)^*\hat w
        =A'(\hat x)^*( E'(\hat b-b_\delta)-\alpha_\delta\hat w).
    \]
    Let $\rho>0$ be as in \cref{ass:nonlinear:energy-noise}.
    By the assumption that $(e_\delta+\delta^q + \delta^p)/\alpha_\delta \downto 0$ and $\alpha_\delta \downto 0$ as $\delta \downto 0$, \cref{lemma:nonlinear:a-convergence}, and \eqref{eq:nonlinear:strong-subreg:u-cond}, for suitably small $\delta>0$, we have $x_\delta \in U_\rho \subset U$.
    Hence by the accuracy estimate in \cref{ass:nonlinear:energy-noise} and the strong local subdifferentiability included in \cref{ass:nonlinear:strong-source-condition},
    \begin{equation}
        \label{eq:nonlinear:strong-subreg-est0}
        \begin{aligned}[t]
        e_\delta
        &
        \ge
        [J_\delta+\alpha_\delta R](x_\delta)-[J_\delta+\alpha_\delta R](\hat x)
        \\
        &
        \ge \dualprod{J_\delta'(\hat x) - \alpha_\delta A'(\hat x)^*\hat w}{x_\delta-\hat x}
        + \gamma\norm{x_\delta-\hat x}_\delta^2
        \\
        &
        = \iprod{E'(\hat b-b_\delta)-\alpha_\delta \hat w}{A'(\hat x)(x_\delta-\hat x)}
        + \gamma\norm{A'(\hat x)(x_\delta-\hat x)}_Y^2
        + \gamma\gamma_\delta\norm{x_\delta-\hat x}_X^2.
        \end{aligned}
    \end{equation}
    Now, for any $\gamma' \in (0, \gamma)$, using Young's inequality twice, we obtain
    \begin{equation}
        \label{eq:nonlinear:est}
        e_\delta
        \ge
        -\frac{1}{4\gamma'}\norm{E'(\hat b-b_\delta)}_Y^2
        - \frac{\alpha_\delta^2}{4(\gamma-\gamma')}\norm{\hat w}^2
        +\gamma\gamma_\delta\norm{x_\delta-\hat x}_X^2.
    \end{equation}
    Thus
    \begin{equation*}
        \norm{x_\delta-\hat x}_X^2
        \le
        \frac{e_\delta}{\gamma\gamma_\delta}
        +\frac{\delta^2}{4\gamma'\gamma\gamma_\delta}
        +\frac{\alpha_\delta^2}{4(\gamma-\gamma')\gamma\gamma_\delta}\norm{\hat w}_Y^2.
    \end{equation*}
    Taking $\gamma'=\frac{1}{2}\gamma$, this yields the claim.
\end{proof}

\begin{corollary}
    \label{cor:nonlinear:strong-subreg}
    Suppose \cref{ass:nonlinear:energy-noise} holds along with the strong source condition of \cref{ass:nonlinear:strong-source-condition} at $\hat x$.
    If
    \[
        \lim_{\delta \downto 0} \frac{1}{\min\{\alpha_\delta, \gamma_\delta\}}(\alpha_\delta^2, \delta^{\min\{2,q,p\}}, e_\delta)=0
    \]
    then
    \[
        \lim_{\delta \downto 0} \norm{x_\delta-\hat x}_X=0.
    \]
\end{corollary}

\subsection{An estimate based on a semi-strong source condition}

The following assumption and theorem adapt \cref{ass:linear:semi-strong-source-condition,lemma:linear:subreg} to non-linear $A$.

\begin{assumption}[Semi-strong source condition; nonlinear case]
    \label{ass:nonlinear:semi-strong-source-condition}
    Assume that $\hat x \in X$ satisfies for some $\hat w \in Y$ the basic source condition of \eqref{eq:nonlinear:basic-source-condition}.
    Moreover, for all $\delta>0$, for given $\alpha_\delta,\gamma_\delta>0$, assume that $J_\delta+\alpha_\delta R$ is $(A'(\hat x), \gamma_\delta)$-strongly locally subdifferentiable at $\hat x$ for $J_\delta'(\hat x)-\alpha_\delta A'(\hat x)^*\hat w$ with respect to the set $\hat X$.
    The factor $\gamma>0$ and neighbourhood $U=U^{\hat x}$ of $(A'(\hat x), \gamma_\delta)$-strong local subdifferentiability must be independent of $\delta>0$.
    Then we say that $\hat x$ satisfies for $\hat w$ the \term{semi-strong source condition}.
\end{assumption}

\begin{lemma}
    \label{lemma:nonlinear:subreg}
    Suppose \cref{ass:nonlinear:energy-noise} and the \emph{semi-strong} source condition of \cref{ass:nonlinear:semi-strong-source-condition} hold at some $\hat x \in X$ for some $\hat w$ with the neighbourhood of strong local subdifferentiability $U^{\hat x} \supset U_\rho$ for some $\rho>0$.
    Suppose $(e_\delta+\delta^q + \delta^p)/\alpha_\delta \downto 0$ and $\alpha_\delta \downto 0$ as $\delta \downto 0$.
    Then there exists $\bar\delta>0$ and  $\gamma=\gamma(\kappa)$ such that if  $\delta \in (0, \bar\delta)$ and $x_\delta \in U$, we have
    \[
        \dist^2(x_\delta, \hat X)
        \le
        \frac{e_\delta}{\gamma\gamma_\delta}
        +\frac{\delta^2}{4\gamma'\gamma\gamma_\delta}
        +\frac{\eta^2\alpha_\delta^2}{4(\gamma-\gamma')\gamma\gamma_\delta}\norm{\hat w}_Y^2.
    \]
\end{lemma}

\begin{proof}
    By \cref{lemma:linear:a-convergence} and the assumptions that $(e_\delta+\delta^q + \delta^p)/\alpha_\delta \downto 0$ and $\alpha_\delta \downto 0$ as $\delta \downto 0$ and $U^{\hat x} \supset U_\rho$, for suitably small $\delta>0$, we have $x_\delta \in U^{\hat x}$.
    As in \eqref{eq:linear:subreg:base-est}, using the assumed $(A,\gamma_\delta)$-strong local subdifferentiability of $J_\delta+\alpha_\delta R$, analogously to \eqref{eq:nonlinear:strong-subreg-est0} we have
    \begin{equation*}
        e_\delta
        \ge
        \iprod{E'(\hat b-b_\delta)-\alpha_\delta\hat w}{A'(\hat x)(x_\delta-\hat x)}
        + \gamma\norm{A'(\hat x)(x_\delta-\hat x)}_Y^2
        + \gamma\gamma_\delta\dist^2(x_\delta, \hat X).
    \end{equation*}
    From here we proceed as in the proof of  \cref{thm:nonlinear:strong-subreg}.
\end{proof}

The next result is proved exactly as \cref{thm:linear:subreg:weak} using \cref{lemma:nonlinear:subreg} in place of \cref{lemma:linear:subreg}.

\begin{theorem}
    \label{thm:nonlinear:subreg:weak}
    Suppose \cref{ass:nonlinear:energy-noise} that there exists a collection $\tilde X \subset \hat X$ of points satisfying the \emph{semi-strong} source condition of \cref{ass:nonlinear:semi-strong-source-condition} such that $\Union_{\hat x \in \tilde X} U^{\hat x} \supset \hat X$.
    Also suppose that $U_\rho$, as defined in \eqref{eq:nonlinear:strong-subreg:u-cond}, is weakly or weakly-$*$ compact for some $\rho>0$, and each $U^{\hat x}$ for all $\hat x \in \tilde X$ correspondingly weakly or weakly-$*$ open.
    If
    \[
        \lim_{\delta \downto 0} \frac{1}{\min\{\alpha_\delta, \gamma_\delta\}}(\alpha_\delta^2, \delta^{\min\{2,q,p\}}, e_\delta)=0
    \]
    then
    \[
        \lim_{\delta \downto 0} \dist(x_\delta, \hat X)=0.
    \]
\end{theorem}

\subsection{Examples}

The following \cref{lemma:nonlinear:to-prove:strong,lemma:nonlinear:to-prove:semistrong} are the counterparts of \cref{lemma:linear:to-prove:strong,lemma:linear:to-prove:semistrong} for nonlinear $A$.
We concentrate for simplicity on $E=\frac{1}{2}\norm{\freevar}_Y^2$.
We need the \term{approximate linearity} condition
\begin{equation}
    \label{eq:nonlinear:approx}
    \frac{1}{2}\norm{A(x)-A(\hat x)}^2
    +\iprod{A(\hat x)-b_\delta}{A(x)-A(\hat x)-A'(\hat x)(x-\hat x)}
    \ge
    \eta\norm{A'(\hat x)(x-\hat x)}^2
    \quad (x \in U_A)
\end{equation}
for some $\eta>0$ and a neighbourhood $U_A$ of $\hat x \in \hat X$. By Pythagoras' three-point identity, this holds with $\eta=\frac{1}{2}$ and $U_A=X$ if $A$ is linear.
Given that \cref{ex:nonlinear:l2} establishes $\norm{A(\hat x)-b_\delta}=\norm{\hat b-b_\delta} \le \delta$ for $E=\frac{1}{2}\norm{\freevar}_Y^2$, \eqref{eq:nonlinear:approx} follows from
\begin{equation*}
    \frac{1}{2}\norm{A(x)-A(\hat x)}^2
    \ge
    \delta\norm{A(x)-A(\hat x)-A'(\hat x)(x-\hat x)}
    +\eta\norm{A'(\hat x)(x-\hat x)}^2
    \quad (x \in U_A).
\end{equation*}

\begin{lemma}
    \label{lemma:nonlinear:to-prove:strong}
    Let $E=\frac{1}{2}\norm{\freevar}_Y^2$ on a Hilbert space $Y$ and suppose $A$ and $R$ are as in \cref{ass:nonlinear:energy-noise}.
    Suppose \eqref{eq:nonlinear:approx} holds at $\hat x$ for a given $\delta>0$.
    Then $J_\delta+\alpha_\delta R$ is strongly locally subdifferentiable at $\hat x$ for $J_\delta'(\hat x)-\alpha_\delta \hat d$ with respect to the norm $\norm{\freevar}_\delta$ if, for corresponding neighbourhood $U \ni \hat x$, $U \subset U_A$ and factor $\gamma>0$,
    \begin{equation}
        \label{eq:nonlinear:to-prove:strong}
        \alpha_\delta[R(x)-R(\hat x)-\dualprod{\hat d}{x-\hat x}]
        +\left(\eta-\gamma\right)\norm{A'(\hat x)(x-\hat x)}_2^2
        \ge
        \gamma\gamma_\delta \norm{x-\hat x}^2
        \quad (x \in U).
    \end{equation}
\end{lemma}

The proof follows from the proof of the next lemma after expanding $\norm{\freevar}_\delta$ and taking $\hat X=\{\hat x\}$.

\begin{lemma}
    \label{lemma:nonlinear:to-prove:semistrong}
    Let $E=\frac{1}{2}\norm{\freevar}_Y^2$ on a Hilbert space $Y$ and suppose $A$ and $R$ are as in \cref{ass:nonlinear:energy-noise}.
    Suppose \eqref{eq:nonlinear:approx} holds at $\hat x$ for a given $\delta>0$.
    Then $J_\delta+\alpha_\delta R$ is $(A, \gamma_\delta)$-strongly locally subdifferentiable at $\hat x$ for $J_\delta'(\hat x)-\alpha_\delta \hat d$ with respect to $\hat X$ if, for corresponding neighbourhood $U \ni \hat x$, $U \subset U_A$ and factor $\gamma>0$,
    \begin{equation}
        \label{eq:nonlinear:to-prove}
        \alpha_\delta[R(x)-R(\hat x)-\dualprod{\hat d}{x-\hat x}]
        +\left(\eta-\gamma\right)\norm{A'(\hat x)(x-\hat x)}_2^2
        \ge
        \gamma\gamma_\delta \dist^2(x, \hat X)
        \quad (x \in U).
    \end{equation}
\end{lemma}

\begin{proof}
    Minding the definition of $J_\delta$, we need to show that
    \begin{multline}
        \label{eq:nonlinear:to-prove:0}
        E(A(x)-b_\delta)-E(A(\hat x)-b_\delta)-\iprod{E'(A(\hat x)-b_\delta)}{A'(\hat x)(x-\hat x)}
        \\
        +\alpha_\delta[R(x)-R(\hat x)-\dualprod{A'(\hat x)^*\hat w}{x-\hat x}]
        \ge
        \gamma\norm{A'(\hat x)(x-\hat x)}_Y^2 + \gamma\gamma_\delta\dist(x, \hat X)^2
        \quad (x \in U).
    \end{multline}
    For $E=\frac{1}{2}\norm{\freevar}_Y^2$, using Pythagoras' identity \eqref{eq:intro:three-point} and the approximate linearity condition \eqref{eq:nonlinear:approx}, we have
    \[
        \begin{aligned}[t]
        E(A(x)-b_\delta)&-E(A(\hat x)-b_\delta)-\iprod{E'(A(\hat x)-b_\delta)}{A'(\hat x)(x-\hat x)}
        \\
        &
        =
        \frac{1}{2}\norm{A(x)-b_\delta}^2 - \frac{1}{2}\norm{A(\hat x)-b_\delta}^2
        - \iprod{A(\hat x)-b_\delta}{A'(\hat x)(x-\hat x)}
        \\
        &
        =
        \frac{1}{2}\norm{A(x)-A(\hat x)}_Y^2
        + \iprod{A(\hat x)-b_\delta}{A(x)-A(\hat x)-A'(\hat x)(x-\hat x)}
        \\
        &
        \ge \eta\norm{A'(\hat x)(x-\hat x)}^2.
        \end{aligned}
    \]
    Applying this in \eqref{eq:nonlinear:to-prove} proves \eqref{eq:nonlinear:to-prove:0}.
\end{proof}

\begin{corollary}[Nonlinear total variation regularised image reconstruction]
    \label{cor:tv:nonlinear}
    Let $E=\frac{1}{2}\norm{\freevar}_Y^2$, $R(x)=\norm{Dx}_{\Meas(\Omega; \R^m)}$, and $A \in C^1(X; Y)$ in $X=L^2(\Omega)$ and a Hilbert space $Y$. Suppose for some $\bar\delta,\rho>0$ that the approximate linearity condition \eqref{eq:nonlinear:approx} holds at $\hat x$ for all $\delta \in (0, \bar\delta)$ with $U_A \supset U_\rho$.
    Also suppose \eqref{eq:nonlinear:accuracy} holds, $\hat x, \hat d \in L^2(\Omega)$ satisfy \eqref{eq:tv-basic-source-cond}, and there exists a corresponding collection $\mathcal{O}$ of strictly flat areas satisfying
    \begin{equation}
        \label{eq:nonlinear:tv:cond}
        K_{\mathcal{O}}^*K_{\mathcal{O}} + A'(\hat x)^*A'(\hat x) \ge \epsilon\Id
        \text{ for some } \epsilon>0.
    \end{equation}
    If the accuracy and regularisation parameters satisfy
    \begin{equation}
        \label{eq:nonlinear:tv:conv}
        \lim_{\delta \downto 0} \left(\alpha_\delta, \frac{\delta^{\min\{2,p,q\}}}{\alpha_\delta}, \frac{e_\delta}{\alpha_\delta}\right)=0,
    \end{equation}
    then $\norm{x_\delta-\hat x}_2 \downto 0$.
    Moreover, \eqref{eq:nonlinear:strong-subreg:estimate} holds for small enough $\delta>0$.
\end{corollary}

\begin{proof}
    Due to \cref{ex:nonlinear:l2} and \eqref{eq:nonlinear:accuracy},  \cref{ass:nonlinear:energy-noise} holds.
    By the assumption $U_A \supset U_\rho$, the approximate linearity condition \eqref{eq:nonlinear:approx} is valid for $\delta>0$ small enough that $x_\delta \in U_\rho$ due to \cref{lemma:nonlinear:a-convergence} and \eqref{eq:nonlinear:tv:conv}.
    Therefore, in the proofs of \cref{thm:tv:main,lemma:tv:strong-subreg}, where $\gamma_\delta=\alpha_\delta$, we simply replace \cref{thm:linear:strong-subreg} by  \cref{thm:nonlinear:strong-subreg}, and  \cref{lemma:linear:to-prove:strong} by \cref{lemma:nonlinear:to-prove:strong}.
\end{proof}

The unconditional Lasso example of \cref{thm:lasso:main} does not extend as readily to nonlinear $A$. However, if we assume that $A'(\hat x)^*A'(\hat x) + \sum_{k \in Z(\hat x, \hat d)} \mathbb{1}_k \mathbb{1}_k^\top \ge \epsilon \Id$ (compare \eqref{eq:lasso:m} and  \eqref{eq:nonlinear:tv:cond}), then it is possible to produce convergence to specific $\hat x$ as in \cref{cor:tv:nonlinear}.

\section{Regularisation complexity of optimisation methods in Hilbert spaces}
\label{sec:complexity}

We now briefly discuss how we can use some popular nonsmooth optimisation methods to construct $x_\delta$ satisfying the accuracy estimate \eqref{eq:linear:accuracy} and the parameter convergence conditions \eqref{eq:linear:strong-subreg:convergence-cond}.
We start with forward-backward splitting, mainly applicable to the $\ell^1$-regularised regression of \cref{thm:lasso:main}, in which case it is also known as iterative soft-thresholding \cite{chambolledevore1998nonlinear,daubechies2004surrogate,wrightnovak2009sparse}.
We then look at the more widely applicable primal-dual proximal splitting (PDPS), also known as the Chambolle--Pock method.
Besides the original references below, the methods and their convergence properties are discussed, for example, in \cite{clasonvalkonen2020nonsmooth,tuomov-proxtest}.
Due to the necessities of effective first-order methods, we need to restrict our attention to Hilbert spaces.

\subsection{Forward-backward splitting}

The forward-backward splitting method of \cite{lionsmercier1979splitting} applies to problems of the form
\[
    \min_{x \in X} F(x) + G(x),
\]
on Hilbert spaces $X$ where $F: X \to \extR$ and $G: X \to \R$ are convex, proper, and lower semicontinuous, and $G$ has an $L$-Lipschitz gradient. Take a step length parameter $\tau>0$ such that $\tau L < 1$ and an initial iterate $x^0 \in X$.
If $F$ has a simple closed-form proximal operator $\prox_{\tau F}(x) \defeq \min_{\alt x} \frac{1}{2}\norm{x-\alt x}+\tau F(\alt x)$, the method iterates
\[
    x^{k+1} \defeq \prox_{\tau F}(x^k-\tau \grad G(x^k)).
\]
Taking $F=J_\delta$ and $G=\alpha_\delta R$, we now apply the method to \eqref{eq:linear:problem:regularised}.

\begin{theorem}
    \label{thm:fb:reg}
    Suppose that \cref{ass:linear:main} and the strong source condition of \cref{ass:linear:strong-source-condition} hold.
    For each $\delta>0$, take $N_\delta$ iterations of forward-backward splitting, starting from the same initial iterate $x^0 \in X$ with the same step length parameter $\tau>0$ satisfying $\tau L<1$.
    Denote the iterates by $\{x_\delta^k\}_{k \in \N}$.
    If
    \begin{equation}
        \label{eq:fb:convergence-cond}
        \lim_{\delta \downto 0} \frac{1}{\min\{\alpha_\delta, \gamma_\delta\}}(\alpha_\delta^2, \delta^2, N_\delta^{-1})=0,
    \end{equation}
    then
    \[
        \lim_{\delta \downto 0} \norm{x^{N_\delta}_\delta-\hat x}_X=0.
    \]
\end{theorem}

\begin{proof}
    The iterates of the forward-backward splitting method are monotone ($[F+G](x^{k+1}) \le [F+G](x^k)$) and satisfy for any $\optx \in X$ the estimate (see, e.g., \cite{beck2017firstorder,clasonvalkonen2020nonsmooth})
    \[
        [F+G](x^N)-[F+G](\optx) \le \frac{1}{2\tau N}\norm{x^0-\optx}^2
        \quad (N \in \N).
    \]
    Therefore, with $F(x)=\frac{1}{2}\norm{Ax-b_\delta}_Y^2$ and $G(x) \defeq \alpha_\delta R(x)$, the accuracy estimate \eqref{eq:linear:accuracy} is satisfied for
    \[
        e_\delta = \frac{1}{2\tau N_\delta}\norm{x^0-\hat x}^2
    \]
    after taking $N_\delta$ iterations from the fixed initial iterate $x^0$.
    Thus the condition \eqref{eq:linear:strong-subreg:convergence-cond} of \cref{cor:linear:strong-subreg} is satisfied by choosing $\alpha_\delta>0$ and $N_\delta$ such that \eqref{eq:fb:convergence-cond} holds.
\end{proof}

In particular, if $\alpha_\delta=\gamma_\delta$ as in the Lasso of \cref{thm:lasso:main}, it suffices to take $\alpha_\delta \downto 0$ and $N_\delta \upto \infty$ such that $\delta^2/\alpha_\delta \downto 0$, and $\alpha_\delta N_\delta \upto \infty$ as $\delta \downto 0$.

\subsection{Primal-dual proximal splitting}

Primal-dual methods, for example the primal-dual proximal splitting (PDPS) of Chambolle and Pock \cite{chambolle2010first}, do not directly provide an accuracy estimate of the type \eqref{eq:linear:accuracy}. They provide estimates on a gap functional. To be more specific, consider the general problem
\begin{equation}
    \label{eq:primal-dual:problem}
    \min_{x \in X} F(x)+G(Kx),
\end{equation}
for convex, proper, lower semicontinuous $F: X \to \extR$ and $G: Y \to \extR$, and $K \in \linear(X; Y)$ on Hilbert spaces $X$ and $Y$.
Writing $G^*$ for Fenchel conjugate of $G$, taking step length parameters $\tau,\sigma>0$ with $\tau\sigma\norm{K}^2 < 1$ and an initial iterate $(x^0, y^0) \in X \times $, the PDPS then iterates
\begin{equation}
    \label{eq:pdps}
    \left\{\begin{array}{l}
        x^{k+1} \defeq \prox_{\tau F}(x^k-\tau K^*y^k), \\
        y^{k+1} \defeq \prox_{\sigma G^*}(y^k+\sigma K(2x^{k+1}-x^k)).
    \end{array}\right.
\end{equation}

Define the Lagrangian gap functional
\[
    \GenGap(x, y; \opt x, \opt y) \defeq \left(F(x)+\iprod{Kx}{\opt y}-G^*(\opt y)\right) - \left(F(\opt x) + \iprod{K\opt x}{y} - G^*(y)\right).
\]
The iterates of the PDPS satisfy for all comparison points $(\alt x, \alt y) \in X \times Y$, for some constant $C>0$ that \cite{tuomov-proxtest,clasonvalkonen2020nonsmooth,he2012convergence}
\begin{equation}
    \label{eq:primaldual:gap-first}
    \frac{1}{2}\norm{(x^N,y^N)-(\optx,\opty)}^2_{M}
    + \sum_{k=0}^{N-1} \gap(x^k, y^k; \alt x, \alt y) \le \frac{1}{2}\norm{(x^0, y^0)-(\alt x, \alt y)}^2_{M}
    \quad (N \in \N),
\end{equation}
where
\[
    \norm{u}_M \defeq \sqrt{\iprod{Mu}{u}}
    \quad\text{and}\quad
    M \defeq \begin{pmatrix} \inv\tau\Id & -K^* \\ -K & \inv\sigma\Id \end{pmatrix}.
\]
We want to develop \eqref{eq:primaldual:gap-first} into a function value estimate to use the regularisation theory of \cref{sec:linear}. For the next lemma, we need to know that by the Fenchel--Rockafellar theorem, minimisers $\opt x \in X$ of \eqref{eq:primal-dual:problem} are characterised by the existence of a \term{primal-dual solution pair} $(\opt x, \opt y) \in X \times Y$ such that
\[
    -K^*\opt y \in \subdiff F(\opt x)
    \quad\text{and}\quad
    K^*\opt x \in \subdiff G^*(\opt y).
\]

\begin{lemma}
    \label{lemma:primaldual:function-value-bound}
    Let $\{(x^k, y^k)\}_{k=1}^\infty$ be generated by the PDPS for the problem \eqref{eq:primal-dual:problem} and initial iterates $(x^0, y^0)$.
    Suppose the step length parameters satisfy $\tau\sigma\norm{K}^2 < 1$.
    Let $(\opt x, \opt y)$ be a primal-dual solution pair.
    Define
    \begin{equation}
        \label{eq:primaldual:c}
        C \defeq \norm{\opt x}_X + \frac{\sqrt{\tau}\norm{(x^0, y^0)-(\opt x, \opt y)}_{M}}{\sqrt{1-\tau\sigma\norm{K}^2}}.
    \end{equation}
    For all $N \in \N$, define the ergodic variables $\alt x^N \defeq \frac{1}{N} \sum_{k=0}^{N-1} x^N$ and  $\alt y^N \defeq \frac{1}{N} \sum_{k=0}^{N-1} y^N$.
    Suppose there exists a bounded set $B_Y \subset Y$ such that
    \begin{equation}
        \label{eq:primaldual:semi-conj}
        \sup_{\alt y \in B_Y}(\iprod{Kx}{\alt y} - G^*(\alt y))=G(Kx)
        \quad (\norm{x} \le C).
    \end{equation}
    Then, for any $\alt x \in X$,
    \[
        F(\alt x^N)+G(K\alt x^N) \le F(\alt x) + G(K\alt x) + \sup_{\alt y \in B_Y} \frac{\norm{(x^0, y^0)-(\alt x, \alt y)}^2_{M}}{2N}.
    \]
\end{lemma}

\begin{proof}
    Since $(\opt x, \opt y)$ is a primal-dual solution pair, we have $\gap(\freevar; \opt x, \opt y) \ge 0$ as a consequence of the Fenchel--Rockafellar theorem.
    Since $\tau\sigma\norm{K} < 1$, Young's inequality shows that $\iprod{Mu}{u} \ge\inv\tau(1-\tau\sigma\norm{K}^2)\norm{x}^2$.
    By \eqref{eq:primaldual:gap-first} we therefore have for all $k \ge \N$ that
    \[
        \inv\tau(1-\tau\sigma\norm{K}^2)\norm{x^k-\opt x}_X^2
        \le
        \norm{(x^0, y^0)-(\opt x, \opt y)}^2_{M}.
    \]
    In other words $\norm{x^k} \le C$, consequently $\norm{\alt x^N} \le C$ for all $N \in \N$.

    As in \cite{tuomov-predict}, we marginalise the gap with respect to the dual variable:
    \[
        \inf_{y \in Y} \GenGap(x, y; \alt x, \alt y)
        =
        \left(F(x)+\iprod{Kx}{\alt y}-G^*(\alt y)\right) - \left(F(\alt x) + G(K\alt x)\right)
        \quad (x, \alt x \in X,\, \alt y \in Y).
    \]
    Using \eqref{eq:primaldual:semi-conj}, it follows that
    \[
        \sup_{\alt y \in B_Y} \inf_{y \in Y} \GenGap(x, y; \alt x, \alt y)
        \ge \left(F(x)+G(Kx)\right) -  \left(F(\alt x) + G(K\alt x)\right)
        \quad (\norm{x} \le C).
    \]
    Jensen's inequality and \eqref{eq:primaldual:gap-first} give the ergodic gap estimate
    \[
        \gap(\alt x^N, \alt y^N; \alt x, \alt y) \le \frac{1}{2N}\norm{(x^0, y^0)-(\alt x, \alt y)}^2_{M}
        \quad (N \in \N).
    \]
    Using that $\norm{\alt x^N} \le C$ and combining these two inequalities, we obtain the claim.
\end{proof}

We now return to the problem \eqref{eq:linear:problem:regularised} with $R=R_0 \circ Q$ for convex, proper and lower semicontinuous $G: Z \to \extR$ and $Q \in \linear(X; Y)$.
We assume that all $X$, $Y$, and $Z$ are Hilbert spaces.
For each $\delta>0$, we define
\begin{equation}
    \label{eq:pdps-formulation}
    F_\delta(x) \defeq 0,
    \quad
    G_\delta(y, z) \defeq \frac{1}{2}\norm{y-b_\delta}^2 + \alpha_\delta R_0(z),
    \quad\text{and}\quad
    Kx \defeq (Ax, Qx).
\end{equation}
Then
\[
    J_\delta + \alpha_\delta R = F_\delta + G_\delta \circ K.
\]

\begin{theorem}
    \label{thm:primaldual:reg}
    Suppose that \eqref{eq:linear:noise} and the strong source condition of \cref{ass:linear:strong-source-condition} hold at $\hat x$ with $R=R_0 \circ Q$ for convex, proper and lower semicontinuous $R_0: Z \to \extR$ and $Q \in \linear(X; Y)$.
    For each $\delta>0$, take $N_\delta$ iterations of the PDPS for the problem $\min_x F_\delta(x)+G_\delta(Kx)$, starting for each $\delta>0$ from the same initial iterate $(x^0, y^0_*, z^0_*) \in X \times Y \times Z$ with the same step length parameters $\tau,\sigma>0$ satisfying $\tau\sigma(\norm{A}^2+\norm{Q}^2)<1$.
    Suppose $\range \subdiff R_0 \defeq \Union_{z \in Z} \subdiff R_0(z)$ is bounded and that $G_\delta \circ K$ is coercive for all $\delta>0$.
    Denote the primal iterates by $\{x_\delta^k\}_{k \in \N}$, and the corresponding ergodic iterates by $\alt x_\delta^N \defeq \frac{1}{N}\sum_{k=0}^{N-1} x_\delta^k$.
    If
    \begin{equation}
        \label{eq:primaldual:convergence-cond}
        \lim_{\delta \downto 0} \frac{1}{\min\{\alpha_\delta, \gamma_\delta\}}(\alpha_\delta^2, \delta^2, N_\delta^{-1})=0,
    \end{equation}
    then
    \[
        \lim_{\delta \downto 0} \norm{\alt x^{N_\delta}_\delta-\hat x}_X=0.
    \]
\end{theorem}

\begin{proof}
    We  use \cref{cor:linear:strong-subreg}, for which we need to verify \eqref{eq:linear:strong-subreg:convergence-cond} for some $e_\delta$ satisfying \eqref{eq:linear:accuracy}. \cref{ass:linear:strong-source-condition} and  \eqref{eq:linear:noise} of \cref{ass:linear:main} we have assumed.
    We do this via \cref{lemma:primaldual:function-value-bound} applied to $F=F_\delta$, $G=G_\delta$, and $K$, but need $C=C_\delta$ defined in \eqref{eq:primaldual:c} to be bounded over $\delta>0$.
    We also need to construct $B_Y$ satisfying \eqref{eq:primaldual:semi-conj} for $G=G_\delta$ for all small enough $\delta>0$.

    We have
    \begin{equation}
        \label{eq:primaldual:g-delta-subdiff}
        \subdiff G_\delta(y, z) \subset \{y-b_\delta\} \times \alpha_\delta \subdiff R_0(z).
    \end{equation}
    Since \eqref{eq:linear:noise} implies that $b_\delta \to b$ as $\delta \downto 0$,  \eqref{eq:primaldual:g-delta-subdiff} and the assumption that $\range \subdiff R_0$ is bounded prove the exists of a bounded set $B_Y$ such that
    \begin{equation}
        \label{eq:primaldual:subdiff-bound}
        \norm{x} \le C \implies \subdiff G_\delta(Kx) \subset B_Y.
    \end{equation}
    By the Fenchel--Young theorem (see, e.g., \cite{ekeland1999convex,clasonvalkonen2020nonsmooth}) we have
    \[
        G_\delta(y, z) + G_\delta^*(y_*, y_*) = \iprod{y}{y_*} + \iprod{z}{y_*}
        \quad\text{when}\quad
        (y_*, z_*) \in \subdiff G_\delta(y, z).
    \]
    Thus \eqref{eq:primaldual:subdiff-bound} proves \eqref{eq:primaldual:semi-conj}.

    For all $\delta>0$, let $\opt x_\delta$ be a minimiser of $F_\delta+G_\delta \circ K$. Such a point exists because we assume $G_\delta \circ K$ to be proper, coercive, and lower semicontinuous.
    By the Fenchel--Rockafellar theorem (see, e.g., \cite{ekeland1999convex,clasonvalkonen2020nonsmooth}), there also exists a dual solution $(\opt y_\delta^*, \opt z_\delta^*)$, i.e., minimiser of $F_\delta^*(-K^*\freevar)+G_\delta^*$.
    now $x_\delta=\opt x_\delta$ satisfies \eqref{eq:linear:accuracy} with $e_\delta=0$.
    On the other hand, \eqref{eq:primaldual:convergence-cond} verifies \eqref{eq:linear:strong-subreg:convergence-cond} for $e_\delta=0$.
    \Cref{thm:linear:strong-subreg} consequently shows for given $\epsilon>0$ and $\delta \in (0, \bar\delta)$ for small enough $\bar\delta>0$ that $\norm{\opt x_\delta}_X \le \norm{\hat x}_X + \epsilon$.
    By the Fenchel--Rockafellar theorem, we have $(\opt y_\delta^*, \opt z_\delta^*) \in \subdiff G_\delta(K\opt x_\delta)$.
    Thus \eqref{eq:primaldual:subdiff-bound} bounds $\norm{\opt y_\delta^*}$ and $\norm{\opt z_\delta^*}$ uniformly over $\delta \in (0, \bar\delta)0$.
    Consequently $C=C_\delta$ defined in \eqref{eq:primaldual:c} is bounded over $\delta \in (0, \bar\delta)$.

    Now \cref{lemma:primaldual:function-value-bound} proves \eqref{eq:linear:accuracy} for all $\delta \in (0, \bar\delta)$ for $x_\delta=\alt x^{N_\delta}_\delta$ and
    \[
         e_\delta \defeq \sup_{\opty \in B_Y} \frac{\norm{(x^0, y^0)-(\hat x, \opty)}^2_{M}}{2N_\delta}.
    \]
    Thus  \eqref{eq:primaldual:convergence-cond} implies \eqref{eq:linear:strong-subreg:convergence-cond}.
    It remains to apply \cref{cor:linear:strong-subreg}.
\end{proof}

In particular, if $\alpha_\delta=\gamma_\delta$, it suffices to take $\alpha_\delta \downto 0$ and $N_\delta \upto \infty$ such that $\delta^2/\alpha_\delta \downto 0$, and $\alpha_\delta N_\delta \upto \infty$ as $\delta \downto 0$.
However, we cannot directly apply \cref{thm:primaldual:reg} to the total variation \cref{thm:tv:main} as it would require Banach spaces that the proximal steps in the PDPS cannot handle\footnote{It is, however, possible, to develop versions based on Bregman divergences; see, e.g., \cite{tuomov-firstorder}.}. Nevertheless, \cref{thm:primaldual:reg} can be applied to discretised problems, as we next numerically demonstrate.

\subsection{Numerical illustration}

We finish by numerically illustrating \cref{thm:primaldual:reg} for total variation deblurring.
We take the $768 \times 512$ pixel “lighthouse” test image from the free Kodak image suite \cite{franzenkodak}, converted to greyscale values in the range $[0, 1]$. This is the ground-truth $\hat x=\hat b$. Instead of the overall noise level $\delta=\delta(\breve\delta)$, we use the pixelwise noise level $\breve\delta$ as our main parameter.

To generate the data $b_{\breve \delta}$, we apply pixelwise Gaussian noise of varying standard deviation $\breve \delta$ to $\hat b$ and apply to the result our forward operator $A$, a convolution with a Gaussian kernel of standard deviation 2 in a window of $7 \times 7$ pixels.
To employ isotropic total variation regularisation, we take as $Q$ a forward-differences discretisation of the image gradient operator (cell width $h=1$), and $R_0=\norm{\freevar}_{2,1}$ as the sum of two-norms of the gradient vectors over each image pixel.
For each $\breve \delta$ we take $\alpha_{\breve \delta}=\breve\delta/2$ based on rough visual inspection. To ensure that $N_{\breve\delta}\alpha_{\breve\delta} \upto \infty$ as $\breve\delta \downto 0$ with $N_{\breve\delta} \upto \infty$ not too fast for numerical computation to become infeasible, and to always take at least 100 iterations, somewhat arbitrarily we choose $N_{\breve \delta}=100+\inv{\alpha_{\breve\delta}}(t \mapsto \log(1+t))^{1000}(\inv{\breve\delta})$ (1000-fold composition of the logarithmic map). We plot this in \cref{fig:tvdenoising:n}.

We apply the PDPS \eqref{eq:pdps} to the functions \eqref{eq:pdps-formulation} for $\breve\delta \in \{f\cdot 10^{-p} \mid f \in \{1, 0.5\},\, p \in \{0,\ldots,8\}\}$. We use zero initialisation and take as the step length parameters $\tau=5/L$ and $\sigma=0.99/(5L)$ for $L$ a numerically computed estimate on the norm of $K$.
We report in \cref{fig:tvdenoising:dist} the normalised distances to the ground truth for $N=N_{\breve\delta}$ and fixed $N=100$ and $N=1000$ iterations.
The figure illustrates within numerical limits the convergence of the iterate $x_{\breve\delta}^N$ for $N=N_{\breve\delta}$ to the ground-truth as $\breve\delta\downto 0$, whereas with a fixed iteration count no convergence is observed.
For further details, our Julia implementation of the experiments and algorithm is available on Zenodo \cite{tuomov-regtheory-code}.

\begin{figure}
    \pgfplotsset{
        xlabel near ticks,
        ylabel near ticks,
        tick label style = {font=\scriptsize},
        every axis label = {font=\footnotesize},
        legend style = {font=\footnotesize},
        label style = {font=\footnotesize},
    }
    \begin{subfigure}{0.5\textwidth}%
\begin{tikzpicture}
    \begin{axis}[%
        width=\linewidth,
        xmode=log,
        ymax=0.22,
        scaled x ticks=false,
        xminorticks=true,
        minor x tick num=1,
        yminorticks=true,
        minor y tick num=3,
        axis x line*=bottom,
        axis y line*=left,
        legend style={legend pos=north east,inner sep=0pt,outer sep=0pt,legend cell align=left,align=left,draw=none,fill=none,font=\scriptsize},
        xlabel={$\breve \delta$},
        x dir=reverse, 
        ylabel={$\frac{\norm{x_{\breve \delta}^N-\hat x}_2}{\norm{\hat x}_2}$},
        y tick label style={/pgf/number format/fixed},
        outer sep=0pt,
        ]

        \addplot [color=Set2-A, line width=1pt] table[x=delta,y=q]{tv-results.txt};
        \addlegendentry{$N=N_{\breve \delta}$}


        \addplot [color=Set2-C, dashed, line width=1pt] table[x=delta,y=q100]{tv-results.txt};
        \addlegendentry{$N=100$}

        \addplot [color=Set2-D, dashed, line width=1pt] table[x=delta,y=q1000]{tv-results.txt};
        \addlegendentry{$N=1000$}

        \addplot [color=Set2-E, dotted, line width=1pt] table[x=delta,y=q_corrupt]{tv-results.txt};
        \addlegendentry{$b_{\breve \delta}$}

    \end{axis}

\end{tikzpicture}%
        \caption{Normalised distance to ground-truth}%
        \label{fig:tvdenoising:dist}%
    \end{subfigure}
    \begin{subfigure}{0.5\textwidth}%
\begin{tikzpicture}
    \begin{axis}[%
        width=\linewidth,
        xmode=log,
        scaled x ticks=false,
        xminorticks=true,
        minor x tick num=1,
        ymode=log,
        yminorticks=true,
        minor y tick num=3,
        axis x line*=bottom,
        axis y line*=left,
        legend style={legend pos=north east,inner sep=0pt,outer sep=0pt,legend cell align=left,align=left,draw=none,fill=none,font=\scriptsize},
        xlabel={$\breve \delta$},
        x dir=reverse, 
        ylabel={$N_{\breve \delta}$},
        outer sep=0pt,
        ]

        \addplot [color=Set2-A, line width=1pt] table[x=delta,y=N]{tv-results.txt};
    \end{axis}

\end{tikzpicture}%
        \caption{Iteration count $N_{\breve \delta}$}%
        \label{fig:tvdenoising:n}%
    \end{subfigure}
    \caption{Illustration of regularisation complexity of PDPS (\cref{thm:primaldual:reg}) for total variation deblurring.
    In (\subref{fig:tvdenoising:dist}) we display the reconstruction quality in terms of the normalised distance to the ground truth after $N$ iterations for fixed $N$ and a choice $N_{\breve \delta}$ satisfying the conditions of the theorem. We also display the quality of the corrupted data $b_{\breve \delta}$. In (\subref{fig:tvdenoising:n}) we plot the chosen iteration count $N_{\breve \delta}$ against the pixelwise noise level $\breve\delta$.
    }
    \label{fig:tvdenoising}%
\end{figure}

 \providecommand{\eprint}[1]{\href{http://arxiv.org/abs/#1}{arXiv:#1}}
  \providecommand{\noopsort}[1]{}
  \providecommand{\eprint}[1]{\href{http://arxiv.org/abs/#1}{arXiv:#1}}


\end{document}